\documentclass[11pt]{amsart}

\usepackage{styles}

\title[Spectral 2-actions, foams, and frames]{Spectral 2-actions, foams, and frames in the spectrification of Khovanov arc algebras}

\author[Dranowski]{Anne Dranowski}
\address{Department of Mathematics \\ 3620 S. Vermont Ave \\KAP 108 \\  Los Angeles, CA 90089}
\email{dranowsk@usc.edu}

\author[Guo]{Meng Guo}
\address{Department of Mathematics\\ Altgeld Hall 105 \\ 1409 W. Green Street \\ Urbana, IL 61801}
\email{mengguo1@illinois.edu}

\author[Lauda]{Aaron Lauda}
\address{Department of Mathematics and Physics \\ 3620 S. Vermont Ave \\ KAP 108  \\ Los Angeles, CA 90089}
\email{lauda@usc.edu}

\author[Manion]{Andrew Manion}
\address{Department of Mathematics \\ 2108 SAS Hall \\ Box 8205 \\ Raleigh, NC 27695}
\email{ajmanion@ncsu.edu}

\date{\today}

\begin{document}

\begin{abstract}
    Leveraging skew Howe duality, we show that Lawson--Lipshitz--Sarkar's spectrification of Khovanov's arc algebra gives rise to 2-representations of categorified quantum groups over $\mathbb{F}_2$ that we call spectral 2-representations.  These spectral 2-representations take values in the homotopy category of spectral bimodules over spectral categories.   We view this as a step toward a higher representation theoretic interpretation of spectral enhancements in link homology.   A technical innovation in our work is a streamlined approach to spectrifying arc algebras, using a set of canonical cobordisms that we call frames, that may be of independent interest.  As a step towards extending these spectral 2-representations to integer coefficients, we also work in the $\mathfrak{gl}_2$ setting and lift the Blanchet--Khovanov algebra to a multifunctor into a multicategory version of Sarkar--Scaduto--Stoffregen's signed Burnside category.
\end{abstract}

\maketitle
 
\tableofcontents

\setcounter{tocdepth}{2}

\section{Introduction}

One of the most exciting directions in link homology is the rapidly emerging study of ``spectrification'' (cf.\ e.g.\ \cite{LS-flow, HKK, LLS-homotopy, LLS-KST, LLS-platform,KW}). Khovanov homology associates a bigraded homology group to a link $L$, but this homology group is not \emph{a priori} the homology or cohomology of some space. While one could construct such a space using a wedge sum of Moore spaces, Lipshitz and Sarkar \cite{LS-flow} approach the problem differently. Given a diagram for $L$, they define a version of the Khovanov complex that in some sense has coefficients in the sphere spectrum $\mathbb{S}$ rather than $\mathbb{Z}$, so that the Khovanov chain groups become free $\mathbb{S}$-modules and the complex is interpreted as a cube-shaped homotopy coherent diagram in the category of spectra. 

Their Khovanov homotopy type is then the homotopy colimit of this complex; it is not, in general, a sum of Moore spaces, and it contains more information than Khovanov homology \cite{LS2}. Indeed, the resulting Steenrod algebra action on Khovanov homology is nontrivial \cite{LS2,SeedSteenrod}, leading to a spectrum-level refinement of Rasmussen's $s$-invariant \cite{LS-sinvariant}. Hence, the spectrum-level invariants are strictly stronger invariants of knots and links. A construction with similar properties was given by Hu--Kriz--Kriz \cite{HKK} and was shown to be equivalent to Lipshitz--Sarkar's in \cite{LLS-homotopy}. Since this foundational work, there has been a great deal done on spectral link homology, defining colored invariants~\cite{Lobb1},   extending these invariants to tangles ~\cite{LLS-KST,LLS-homotopy,LLS-func} based on spectral versions of Khovanov's arc algebra $H^m$ and its relatives,\footnote{The spectral version of $H^m$ can be viewed in some sense as a $\mathbb{S}$-coefficient version of $H^m$; the spectral bimodules over $H^m$ for tangles are then homotopy colimits of cubical diagrams of bimodules that (disregarding bimodule structure) are free over $\mathbb{S}$.} and defining spectra for $\mathfrak{sl}_n$ link homology~\cite{Lobb2}. 

Since its inception, link homology has always provided a guide to higher representation theory. Even beyond geometric representation theory where the first hints of link homology theories and categorification of quantum invariants could be seen, higher representation theory and link homologies have shared a symbiotic relationship with advances in one area leading to advances in the other. We propose that spectrification in link homology is a strong indication that spectrum-level refinements should exist in higher representation theory.

\subsection{Spectral 2-actions}

In this article, we provide evidence that spectral refinements do indeed exist in higher representation theory. Our starting point is the fact that Khovanov's arc algebra $H^m$ categorifies the space of invariants of even tensor powers of the fundamental representation $V$ of $U_q(\mathfrak{sl}_2)$, i.e.\
\[
    K_0(H^m{\rm -pmod}) \cong {\rm Inv}_{U_q(\mathfrak{sl_2})}(V^{\otimes 2m})\,.
\]
where $H^m{\rm -pmod}$ denotes the category of finitely generated projective modules over $H^m$.
If $\vec{k} = (k_1,\ldots,k_n)$ has $0 \leq k_i \leq 2$ for all $i$, and $2m\le n$ entries of $\vec{k}$ are equal to one, then $V^{\otimes 2m}$ is isomorphic to the $U_q(\mathfrak{sl}_2)$ representation arising (by restriction) from the $U_q(\mathfrak{gl}_2)$ representation 
\begin{equation}\label{eq:glnwtk}
    \textstyle \bigwedge_q^{k_1} V \otimes \cdots \otimes \bigwedge_q^{k_{n}} V.
\end{equation}

Quantum skew Howe duality \cite[Theorem 4.2.2]{ckm-webs} involves an action of $U_q(\mathfrak{gl}_{n})$ on the direct sum of these tensor-of-wedge representations of $U_q(\mathfrak{gl}_2)$. Indeed, by \cite[Theorem 4.2.2(4)]{ckm-webs}, the natural action of $U_q(\mathfrak{gl}_{n})$ on 
\[
    \textstyle \bigwedge^\bullet_q \left( \mathbb{C}_q^2 \otimes \mathbb{C}_q^{n} \right) 
    \cong 
    \bigwedge^\bullet_q \left(\mathbb{C}_q^2 \oplus \cdots \oplus \mathbb{C}_q^2\right)
\]
is such that the weight space in $\mathfrak{gl}_{n}$ weight $\vec{k}$ (say with $2m$ entries equal to one) is the tensor-of-wedges in \Cref{eq:glnwtk} above. The Chevalley generators $E_i$ and $F_j$ of $U_q(\mathfrak{gl}_{n})$ act by $U_q(\mathfrak{gl}_2)$-intertwining maps on the tensor-of-wedge representations; as $U_q(\mathfrak{sl}_2)$-intertwining maps these are the maps 
\[
    V^{\otimes 2m} \to V^{\otimes (2m \pm 2)}
\] 
associated to basic cap and cup flat tangles from $2m$ endpoints to $2m \pm 2$ endpoints. Since these maps intertwine the $U_q(\mathfrak{sl}_2)$ actions, they restrict to maps on $U_q(\mathfrak{sl}_2)$-invariant subspaces $\mathrm{Inv}(V^{\otimes 2m}) \to \mathrm{Inv}(V^{\otimes (2m \pm 2)})$. Moreover, the cup and cap maps yield a representation\footnote{Using \cite[Theorem 4.2.2(3)]{ckm-webs}, one can show that this representation is the direct sum of the irreducible $U_q(\mathfrak{gl}_n)$-representations with highest weight $2 \omega_p$ for $0 \leq p \leq n$.} of $U_q(\mathfrak{gl}_{n})$ on the sum of vector spaces $\mathrm{Inv}(V^{\otimes 2m})$ for $2m \leq n$ even. 

It was first observed by Brundan--Stroppel \cite[Remark 5.7]{BrSt3} that the above structure can be categorified using categorified quantum groups\footnote{The construction is most natural using the $\mathfrak{gl}_{n}$ variant $\cal{U}_Q(\mathfrak{gl}_{n})$ first introduced in \cite{MackaayStosicVaz}.  This variant is related to the $\mathfrak{sl}_n$ version from \cite{KL3} via explicit rescaling 2-functors in \cite{L-param}.  } $\cal{U}_Q(\mathfrak{gl}_{n})$. Specifically, the skew Howe representation is categorified by a 2-representation of the 2-category $\cal{U}_Q(\mathfrak{gl}_{n})$ in which the 1-morphisms $\cal{E}_i$ and $\cal{F}_j$ act as cup or cap bimodules over Khovanov's arc algebras $H^m$, and in which 2-morphisms act as dotted-cobordism maps between these bimodules. Due to a sign discrepancy that we will discuss soon, one must either fix the signs in some way (as in \cite{BrSt3}) or define this 2-representation over $\mathbb{F}_2=\mathbb Z/2\mathbb Z$. We want to study spectrifications of this ``skew Howe'' 2-representation involving $H^m$ as an inroads to spectral higher representation theory more generally.  

At the spectral level one could expect that the Lawson--Lipshitz--Sarkar spectral arc algebras $\mathscr{H}^m$ (or relatives) should appear as the object-level data of a skew Howe spectral 2-representation of $\cal{U}_Q(\mathfrak{gl}_{n})$; spectral $\mathscr{H}^m$-bimodules for caps and cups should appear as the 1-morphism-level data, and spaces of 2-morphisms on the target side should be abelian groups or $\mathbb{F}_2$-vector spaces of morphisms in the stable homotopy category of $\mathscr{H}^m$-bimodules. In particular, the 2-morphism spaces should admit maps from abelian groups or $\mathbb{F}_2$-vector spaces of 2-morphisms in $\cal{U}_Q(\mathfrak{gl}_{n})$; see \Cref{sec:BeyondHtpyCategories} below for a brief discussion of more homotopy coherent notions of 2-representation. 

Indeed, adapting up-to-sign naturality results of Lawson--Lipshitz--Sarkar \cite{LLS-func} from the setting of tangles and tangle cobordisms to the categorified quantum group $\cal{U}_Q(\mathfrak{gl}_n)$, we define such a spectrified 2-representation after tensoring the abelian groups of 2-morphisms with $\mathbb{F}_2$. Let $n \geq 2m$ and let $\cal{U}^{\mathbb{F}_2}_Q(\mathfrak{gl}_n)$ denote the result of tensoring the 2-morphism groups in $\cal{U}_Q(\mathfrak{gl}_n)$ with $\mathbb{F}_2$.
\begin{maintheorem}\label{thm:IntroF2Theorem}
    There is an $\mathbb{F}_2$-linear 2-functor from $\cal{U}^{\mathbb{F}_2}_Q(\mathfrak{gl}_n)$ to the 2-category whose objects are spectral categories, whose 1-morphisms are spectral bimodules, and whose 2-morphism spaces are abelian groups of bimodule maps in the stable homotopy category (i.e.\ with weak equivalences of spectral bimodules inverted), tensored with $\mathbb{F}_2$. This 2-functor recovers $\mathscr{H}^m$ for $2m \leq n$ on objects and Lawson--Lipshitz--Sarkar flat-tangle spectral bimodules over $\mathscr{H}^m$ on 1-morphisms. 
\end{maintheorem}

We actually derive \Cref{thm:IntroF2Theorem} from a slightly more involved statement in which $\mathbb{F}_2$ does not appear (and $2 \neq 0$) but in which the defining relations in $\cal{U}_Q(\mathfrak{gl}_n)$ hold only up to possible sign modifications on each term in each relation.

\subsection{Signs and \texorpdfstring{$\mathfrak{gl}_2$}{gl2}-foams}

To lift \Cref{thm:IntroF2Theorem} to a $\mathbb{Z}$-linear statement, one must deal with a well-known sign discrepancy: while one would want the nilHecke relation among 2-morphisms in $\cal{U}_Q(\mathfrak{gl}_n)$ to follow from the Bar--Natan neck cutting relation for dotted cobordisms that holds on the $H^m$ side, the signs in these relations do not agree~\cite{lau-skew}. This discrepancy is one explanation why Khovanov's original construction of link homology is functorial only up to sign under link cobordisms~\cite{Jac,Kh-cob}. 

\begin{figure}
    \includegraphics{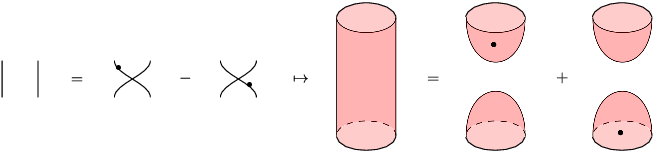}
    \caption{The image of the nilHecke relation in the categorified quantum group maps via Howe duality to Bar-Natan's neck cutting relation, but the signs do not match. }
    \label{fig:nil2neck}
\end{figure}

From the perspective of $\cal{U}_Q(\mathfrak{gl}_n)$ and its relationships with foam and web categories, a satisfying fix for the up-to-sign functoriality of Khovanov homology for closed links is due to Blanchet \cite{Blan} and can be viewed as substituting $\mathfrak{gl}_2$-webs and $\mathfrak{gl}_2$-foams modulo $\mathfrak{gl}_2$-foam relations for crossingless matchings and dotted cobordisms modulo Bar-Natan relations. Foams in the $\mf{gl}_2$ setting carry additional 2-labelled facets not present in the $\mf{sl}_2$ or Bar-Natan setting; see Figure~\ref{fig:nil2neck-blanchet}.

Blanchet's homology for closed links has been spectrified by Krushkal--Wedrich \cite{KW}, using Sarkar--Scaduto--Stoffregen's signed Burnside category $\mathscr{B}_{\sigma}$. Furthermore, based on a careful analysis of the signs in \cite{BrSt3}, Ehrig--Stroppel--Tubbenhauer \cite{EST1} have defined a sign-fixed version of $H^m$ called the Blanchet--Khovanov algebra $\mathfrak{A}^{\mathfrak{F}}$ and established an isomorphism between $\mathfrak{A}^{\mathfrak{F}}$ and an algebra built from $\mathfrak{gl}_2$-webs and $\mathfrak{gl}_2$-foams. However, the Blanchet--Khovanov algebra and its bimodules have not been spectrified.

\begin{figure}
    \includegraphics{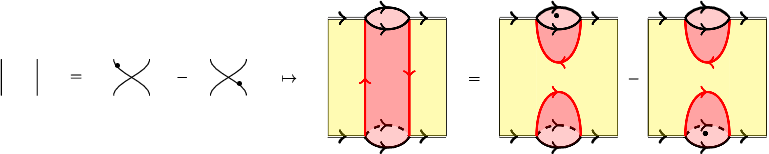}
    \caption{The image of the nilHecke relation in Blanchet foams.  The sign in the nilHecke relation is preserved.  Sliding a dot through a 2-labelled facet multiplies the diagram by $(-1)$. }
    \label{fig:nil2neck-blanchet}
\end{figure}

To see the difficulties involved in spectrifying Blanchet--Khovanov algebras, note that much of the work involved in associating a spectrum to a link diagram can be roughly broken into two steps. The first involves constructing some type of combinatorial data associated with the link. This can take the form of a framed flow category (Cohen--Jones--Segal formalism \cite{LS-flow} as in Floer homotopy theory) or a multicategory enriched in groupoids (Elmendorf--Mandell formalism \cite{HKK}). The second stage utilizes general machinery for producing spectra from these inputs; both approaches lead to the same invariants extending Khovanov homology for closed links \cite{LLS-homotopy}. 

When spectrifying tangle invariants, the Elmendorf--Mandell formalism has proven to be especially useful due to the multiplicativity of Elmendorf--Mandell's infinite loop space machine. Lawson--Lipshitz--Sarkar spectrify $H^m$ and its relatives by defining a multifunctor from a certain ``shape'' multicategory into a multicategory analogue $\underline{\mathscr{B}}$ of the Burnside category $\mathscr{B}$ of the trivial group, which in turn admits a natural multifunctor to the multicategory $\underline{\mathrm{Permu}}$ of permutative categories. The Elmendorf--Mandell construction of $K$-theory spectra then gives a multifunctor from $\underline{\mathrm{Permu}}$ into symmetric spectra; the multifunctoriality of this construction is necessary when using it to define spectral algebras and bimodules.

One difficulty in spectrifying the Blanchet--Khovanov algebras is that it is not clear to us how to get a multifunctor with the right properties (variant of the Elmendorf--Mandell machine) from our multicategory analogue $\underline{\mathscr{B}}_{\sigma}$ of $\mathscr{B}_{\sigma}$ into symmetric spectra or some other symmetric monoidal category of highly structured spectra. We will not address this question in the current paper; instead, we will focus on the step that maps the shape multicategory into $\underline{\mathscr{B}}_{\sigma}$. The following theorem is proved throughout the course of Definition~\ref{def:MultifunctorGl2Case}.

\begin{maintheorem}\label{mthm:burn-bkalg}
    For a finite sequence $\vec{k}$ of elements of $\{0,1,2\}$, let $\mathcal{S}_{\vec{k}}$ be the shape multicategory defined in \Cref{def:ShapeMulticatForWk}. 
    There is an associated multifunctor
    \[
        \Phi_{\vec{k}} \colon \mathcal{S}_{\vec{k}} \to \underline{\mathscr{B}}_{\sigma}.
    \]
    The composition of $\Phi_{\vec{k}}$ with the forgetful map $\underline{\mathscr{B}}_{\sigma} \to \underline{\mathsf{Ab}}$ is identified with the Blanchet--Khovanov algebra $\mathfrak{A}^{\mathfrak{F}}$ associated to $\vec{k}$.
\end{maintheorem}

\begin{mainconjecture}\label{conj:MultifunctorSignedBurnsideToSpectra}
    There is a multifunctor from $\underline{\mathscr{B}}_{\sigma}$ to a multicategory of spectra whose composition with the multifunctor of \Cref{mthm:burn-bkalg} produces spectral algebras relating to the Blanchet--Khovanov algebras in the same way that $\mathscr{H}^m$ relates to $H^m$.
\end{mainconjecture}

\Cref{mthm:burn-bkalg} readily generalizes to bimodules over Blanchet--Khovanov algebras for $\mathfrak{gl}_2$-webs; see \Cref{rem:SignedBurnsideWebBimodules}. \Cref{conj:MultifunctorSignedBurnsideToSpectra} generalizes correspondingly, and we expect that the spectral bimodules of the generalized conjecture recover Krushkal--Wedrich's spectrified $\mathfrak{gl}_2$-foam homology for closed links. 

\subsection{Signed Burnside lift of morphism categories}\label{sec:SignedBurnsideMorphismLift}

To define more homotopical types of 2-representation of $\cal{U}_Q(\mathfrak{gl}_n)$ than in \Cref{thm:IntroF2Theorem} or its proposed $\mathbb{Z}$-linear lift via \Cref{mthm:burn-bkalg}, one would like to have a version of $\cal{U}_Q(\mathfrak{gl}_n)$ with a spectrum of 2-morphisms between any two given 1-morphisms. One would then look for maps on 2-morphism spectra rather than homomorphisms between abelian groups of 2-morphisms. We thus would like a spectrified version of $\cal{U}_Q(\mathfrak{gl}_n)$, where each morphism category of $\cal{U}_Q(\mathfrak{gl}_n)$ becomes a spectral category.

In fact, the Blanchet--Khovanov algebras are a special case of the morphism categories in a $\mathfrak{gl}_2$-foam 2-category $\mathfrak{F}$ (see e.g. \cite[Definition 2.17]{EST1}) which is closely related to $\cal{U}_Q(\mathfrak{gl}_n)$. While we do not yet have a spectral version of $\mathfrak{F}$, our signed Burnside versions of the Blanchet--Khovanov algebras do extend to the more general morphism categories in $\mathfrak{F}$.

\begin{maintheorem}[cf.\ \Cref{rem:GeneralMorSpaces}]
    \label{thm:SignedBurnsideMorphismCategories}
    Generalizing \Cref{mthm:burn-bkalg}, given $\vec{k}_1$ and $\vec{k_2}$, there is a shape multicategory $\cal{S}_{\vec{k}_1,\vec{k}_2}$ and a multifunctor from $\cal{S}_{\vec{k}_1,\vec{k}_2}$ into $\underline{\mathscr{B}}_{\sigma}$. The composition with the forgetful map to abelian groups is identified with the morphism category from $\vec{k}_1$ to $\vec{k}_2$ in $\mathfrak{F}$.
\end{maintheorem}

\begin{remark}
    Morphism categories in $\mathfrak{F}$ are related but not identical to morphism categories in $\cal{U}_Q(\mathfrak{gl}_n)$. By \cite[Proposition 3.22]{QR} they are equivalent, via foamation functors, to morphism categories in a variant called $\check{\cal{U}}_Q(\mathfrak{gl}_{\infty})^{(2)}$. This variant has additional 1-morphisms for divided powers, while the superscript $(2)$ means that one takes the quotient by all objects except those corresponding to sequences $\vec{k}$ with all entries in $\{0,1,2\}$. Replacing $\mathfrak{gl}_n$ by $\mathfrak{gl}_{\infty}$ means that one lets $\vec{k}$ be any infinite sequence $\{k_i\}_{i=1}^{\infty}$ that is eventually zero, rather than fixing the length of $\vec{k}$ to be some $n$. There is an evident 2-functor from $\cal{U}_Q(\mathfrak{gl}_n)$ to $\check{\cal{U}}_Q(\mathfrak{gl}_{\infty})^{(2)}$ for any $n$.
\end{remark}

\subsection{Frames}

We prove \Cref{mthm:burn-bkalg} and \Cref{thm:SignedBurnsideMorphismCategories} using a streamlined version, based on certain canonical multimerge cobordisms called frames, of Lawson--Lipshitz--Sarkar's lift of $H^m$ to the Burnside category. We also use these frames to reformulate Lawson--Lipshitz--Sarkar's construction itself without bringing in foams; this reformulation may be of independent interest.

In more detail, Lawson--Lipshitz--Sarkar \cite{LLS-KST} define a multifunctor from a shape multicategory $\cal{S}_m$ to $\underline{\mathscr{B}}$ in two steps. They first define a multifunctor from $\cal{S}_m$ to a multicategory that encodes a more structured version of embedded cobordisms where the cobordisms are now equipped with extra structure called divides. They then postcompose this multifunctor with a multifunctor $\underline{V}_{\mathrm{HKK}}$ from the divided cobordism multicategory to $\underline{\mathscr{B}}$. The existence of complicated isotopies in general categories of embedded cobordisms motivates their use of divided cobordisms, for which the set of allowed isotopies is more limited.

To avoid isotopies altogether, one could try to define a multifunctor directly from $\cal{S}_m$ to $\underline{\mathscr{B}}$, skipping embedded cobordism categories and their isotopies. From the perspective on $H^m$ taken in \cite{LLS-KST}, however, the definition of multiplication in $H^m$ and thus its spectral version depends on choices of ``multi-merge cobordisms'' in a way that makes it difficult to go directly from $\cal{S}_m$ to $\underline{\mathscr{B}}$. The divided cobordism category plays a useful role in managing the dependence on choices of multi-merge cobordisms.

As discussed e.g.\ in \cite{EST1}, there is an alternate approach to $H^m$ that does not require any choices to be made in defining its multiplication. The ``natural'' web algebras ``$\cal{W}^{\natural}$'' of \cite[Definition~2.19]{EST1}, defined without choices of multi-merge cobordisms, are isomorphic to the web algebras ``$\cal{W}$'' of \cite[Definition~2.24]{EST1} defined with choices of multi-merge cobordisms, by \cite[Lemma~2.27]{EST1}. Basis elements in algebras like $\cal{W}^{\natural}$ look like cobordisms with corners, and they are multiplied by stacking a disjoint union left-to-right rather than gluing a multi-merge cobordism below it (see \Cref{fig:NaturalMult}). 

\begin{figure}[ht!]
    \centering
    \includegraphics{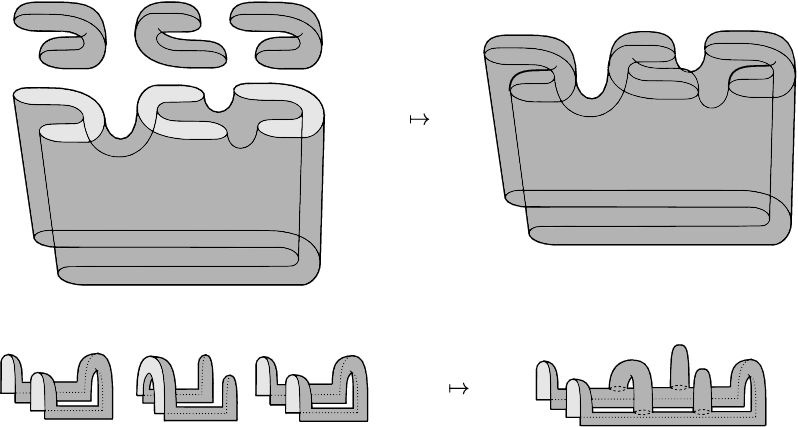}
    \caption{Top: the multi-merge cobordism or ``$\cal{W}$ in \cite{EST1}'' approach to multiplication on $H^m$. Bottom: the left-to-right or ``$\cal{W}^{\natural}$ in \cite{EST1}'' approach to multiplication on $H^m$.}
    \label{fig:NaturalMult}
\end{figure}

A main obstacle, then, to adapting Lawson--Lipshitz--Sarkar's spectrification of $H^m$ to the $\cal{W}^{\natural}$ perspective is that their spectrification makes explicit use of multi-merge cobordisms, e.g.\ via checkerboard colorings of regions in the space where the cobordism is embedded (in the definition of $\underline{V}_{\mathrm{HKK}}$).

We overcome this obstacle by identifying, in the $\cal{W}^{\natural}$ perspective, canonical embedded cobordisms $F$ that become multi-merge cobordisms after translation to the $\cal{W}$ perspective. We call these cobordisms $F$ ``frames;'' see Figure~\ref{fig:NaturalMultUsingFrames}. Whenever Lawson--Lipshitz--Sarkar's construction makes use of a multi-merge cobordism, we can use the corresponding frame $F$ instead. We get the following result, stated more precisely as \Cref{thm:Sl2MultifunctorValid} below.

\begin{figure}[ht!]
    \centering
    \includegraphics{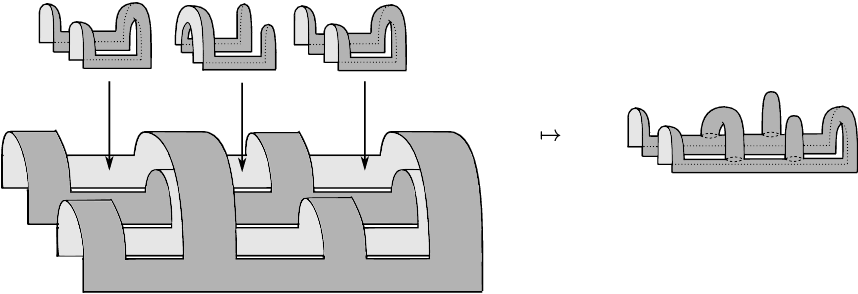}
    \caption{The left-to-right or ``$\cal{W}^{\natural}$ in \cite{EST1}'' approach to multiplication on $H^m$ interpreted as gluing into a frame; the frame is the piece on the bottom left.}
    \label{fig:NaturalMultUsingFrames}
\end{figure}

\begin{maintheorem}\label{thm:DirectlyToBurnside}
    The multifunctor directly from $\mathcal{S}_m$ to $\underline{\mathscr{B}}$ that we define, by unwinding Lawson--Lipshitz--Sarkar's definitions and translating to the frames perspective, agrees with Lawson--Lipshitz--Sarkar's multifunctor used in defining $\mathscr{H}^m$ and can be shown to be well-defined using only the subset of the arguments of \cite{LLS-KST} that we describe in Section~\ref{sec:Frames}.
\end{maintheorem}

We will actually state and prove \Cref{thm:DirectlyToBurnside} more generally, with $\mathcal{S}_m$ replaced by Lawson--Lipshitz--Sarkar's larger multicategory $\underline{2}^N \tilde{\times} {_{m'}}\!\mathcal{T}_m$ used to construct spectral bimodules for tangles. In this way, we also recover the shape to Burnside step of Lawson--Lipshitz--Sarkar's spectrification of tangle bimodules over $H^m$. 

\begin{remark}\label{rem:SignedBurnsideWebBimodules}
    As with \Cref{thm:DirectlyToBurnside}, we could generalize \Cref{mthm:burn-bkalg} to define signed Burnside versions of the bimodules over Blanchet--Khovanov algebras associated to $\mathfrak{gl}_2$-webs, although we will not write out the details since they are parallel to \Cref{thm:DirectlyToBurnside} (see \Cref{rem:WebBimodules}).
\end{remark}

\subsection{Beyond homotopy categories}
\label{sec:BeyondHtpyCategories}

\Cref{thm:IntroF2Theorem} and \Cref{conj:MultifunctorSignedBurnsideToSpectra} are about 2-representations into 2-categories of spectral bimodules and bimodule maps up to homotopy. It is natural to ask if this homotopical quotient can be removed and refined to a more robust spectral representation theory. As mentioned in \Cref{sec:SignedBurnsideMorphismLift}, this would require new spectral extensions of the categorified quantum group.
However, we believe that there are numerous indications that such spectral modifications should exist.  Below we list some of these:

\subsubsection{Steenrod structures in categorified quantum groups}

Initial investigation of Steenrod structures on categorified quantum groups has already been done by Beliakova--Cooper~\cite{CoopB} and by Kitchloo~\cite{Kit}.

\subsubsection{2-Verma Modules}

Another fundamental issue in higher representation theory is the fact that 2-Kac Moody algebras such as $\cal{U}_Q(\mathfrak{gl}_n)$, as currently defined, do not admit 2-representations that categorify infinite-dimensional representations.  This is a fairly serious problem since it excludes the categorification of Verma modules, which are perfectly sensible objects of category $\mathcal{O}$.

Naisse and Vaz ~\cite{NVaz,NVaz2} overcome this issue in the case of $\mf{sl}_2$ by omitting the biadjointness condition, enhancing the nilHecke algebra to the extended nilHecke algebra, and altering the main $\mathfrak{sl}_2$ relation to a non-split short exact sequence rather than a direct sum decomposition.  A key point is that these categories \emph{do not} admit actions by categorified quantum groups.   Instead, this work suggests the existence of a refinement of categorified quantum groups formulated in the setting of triangulated categories.

\subsubsection{Tensor products of 2-representations} 

In Rouquier's work on higher tensor products \cite{RouquierHopfCategories}, one must also replace the usual additive-category setting for categorified quantum groups with an $A_{\infty}$ or $\infty$-categorical setting. Stricter versions are only possible in some special cases ~\cite{manion2020higher,McM1,McM2}. We view this phenomenon as further evidence that the higher representation theory of categorified quantum groups, properly understood, is homotopical in nature. 

\subsubsection{Higher categories of Soergel bimodules}

Recently a preprint was posted by Liu, Mazel-Gee, Reutter, Stroppel, and Wedrich \cite{LMGRSW} constructing a braided monoidal $(\infty,2)$-category version $\mathbf{K}^b_{\mathrm{loc}}(\mathrm{SBim})$ of the derived Soergel bimodule 2-category; in particular, they show that the construction of Rouquier complexes of Soergel bimodules for braids is natural in a fully homotopy coherent sense. The higher representation theory of the Soergel bimodule 2-category is closely related to that of $\cal{U}_Q(\mathfrak{gl}_n)$ \cite{MackaayStosicVaz}, and we would expect that similar homotopy-coherent structures exist in the $\cal{U}_Q(\mathfrak{gl}_n)$ setting.

\subsection{Further remarks}

\begin{remark}\label{rem:SuppressQuantumGradings}
    While we do not discuss quantum gradings explicitly in this paper, the discussion would be parallel to \cite[Section 6]{LLS-KST}; the ideas in that section apply equally well to our constructions here.
\end{remark}

\begin{remark}
    The reason the signed Burnside multicategory appears when trying to spectrify $\mathfrak{gl}_2$-foam constructions is as follows. In Lawson--Lipshitz--Sarkar's spectrification of $H^m$, an important property is that in the usual basis, the structure constants of the algebra multiplication are all nonnegative integers. When passing from $H^m$ to the Blanchet--Khovanov algebra, this feature is not preserved, and negative numbers appear as structure constants. An analogue in the case of closed link homology is Sarkar--Scaduto--Stoffregen's spectrification of odd Khovanov homology, in which the structure constants for the differential in the chosen basis involve both positive and negative numbers. They deal with the minus signs by introducing the signed Burnside category $\mathscr{B}_{\sigma}$. From $\mathscr{B}_{\sigma}$, they get to spectra using explicit constructions into which they can insert various orientation-reversing maps as needed.

    When compared with spectrifying homology groups for closed links, where there are several equivalent approaches, the case of spectrifying algebras is more delicate due to the subtleties of multiplicative infinite loop space theory and symmetric monoidal structures on categories of structured spectra. It does not seem easy to avoid technology like the Elmendorf--Mandell construction or to insert orientation-reversing maps by hand while preserving associativity of the algebra multiplication. On the other hand, it also does not seem easy to adapt the Elmendorf--Mandell construction so that we can work with $\underline{\mathscr{B}}_{\sigma}$ analogously to how Lawson--Lipshitz--Sarkar work with $\underline{\mathscr{B}}$. Optimistically, we hope to navigate these difficulties by working with $\mathbb{Z}/2\mathbb{Z}$-equivariant spectra, but it also seems plausible that a more direct approach exists and we just didn't see it. If so, it would be great to find it!
\end{remark}

\subsection{Outline}

The layout of the body of the paper is roughly in reverse order with respect to the introduction, since we use frames as a tool in multiple places. In \Cref{sec:Preliminaries}, we review the relevant preliminaries, mostly following \cite{LLS-KST}. In \Cref{sec:Frames}, we introduce frames and prove \Cref{thm:DirectlyToBurnside}. In \Cref{sec:BlanchetKhovanovAlgebras}, we review the algebraic definition of Ehrig--Stroppel--Tubbenhauer's Blanchet--Khovanov algebras \cite{EST1}. In \Cref{sec:SignedBurnsideLift}, we define signed Burnside versions of these algebras and of more general morphism categories in $\cal{U}_Q(\mathfrak{gl}_n)$, proving \Cref{mthm:burn-bkalg} and \ref{thm:SignedBurnsideMorphismCategories}. Finally, in \Cref{sec:UpToSign2Action}, we adapt results from \cite{LLS-func} to prove \Cref{thm:IntroF2Theorem}.

\subsection{Acknowledgments}

We would like to thank Robert Lipshitz, Sucharit Sarkar, Hiro Lee Tanaka, and Daniel Tubbenhauer for useful discussions.  A.D.L.\ is
partially supported by NSF grants DMS-1902092 and DMS-2200419, the Army Research Office
W911NF-20-1-0075, and the Simons Foundation collaboration grant on New Structures in Low-dimensional Topology. A.D.\ is partially supported by NSERC. A.M.\ is partially supported by NSF grant number DMS-2151786.

\section{Preliminaries}\label{sec:Preliminaries}

\begin{definition}[cf. Definition 2.1 of \cite{LLS-KST}]
    A \new{multicategory} (or \new{non-symmetric colored operad}) $\uC$ consists of
    \begin{itemize}
        \item a class $\uC_0$ whose elements are called the \emph{objects} of $\uC$;
        \item for any $x_1,\dots,x_n, y\in\uC_0$, a set $\uC(x_1,\dots,x_n;y)$ called the set of \emph{multimorphisms} from $(x_1,\ldots,x_n)$ to $y$;
        \item for any $x^i_j, y_i, z\in\uC_0$, a function 
        \[
            \uC(x_1^1,\dots,x_{n_1}^1;y_1)\times\cdots\times \uC(x_1^m,\dots,x_{n_m}^m;y_m) \times \uC(y_1,\ldots,y_m;z) \to \uC(x^1_1,\dots,x^m_{n_m};z)
        \]
        called \emph{multicomposition};
         \item when $n =1$ and $y = x_1$, a distinguished element  $1_{x_1}\in\uC(x_1;x_1)$ called the \emph{identity} on $x_1$
    \end{itemize}
    such that the associativity and unitality properties (M-5)--(M-7) of \cite[Definition 2.1]{LLS-KST} are satisfied. A \new{multifunctor} $\Phi$ from $\uC$ to $\underline{\cal{D}}$ is an assignment of an object $\Phi(x)$ of $\underline{\cal{D}}$ to each object $x$ of $\uC$ and a function
    \[
        \uC(x_1,\ldots,x_n;y) \to \underline{\cal{D}}(\Phi(x_1),\ldots,\Phi(x_n);\Phi(y))
    \]
    to each tuple of objects $x_1,\ldots,x_n,y$ of $\uC$, such that multicomposition and identities are strictly respected.
\end{definition}

\begin{definition}[cf. Definition 2.2 of \cite{LLS-KST}]\label{def:ShapeMulticat}
    For a finite set $X$, the \new{shape multicategory} of $X$, denoted $\mathcal{S}^0(X)$, has set of objects $X \times X$. Given elements $(x_0,x_1),\ldots,(x_{n-1},x_n)$ of $X \times X$ ($n \geq 0$), there is a unique multimorphism
    \[
        ((x_0,x_1),\ldots,(x_{n-1},x_n)) \to (x_0,x_n),
    \]
    and these are the only multimorphisms in $\mathcal{S}^0(X)$. Multicomposition is defined in the only possible way. The unique multimorphism from an object $(x_0,x_1)$ to itself is an identity multimorphism for multicomposition. 
\end{definition}

\begin{example}
    Let $A$ be an algebra with orthogonal set of idempotents $I = \{e_x\}_{x\in X}$. Associated to $A$ is a multifunctor $F : \mathcal S^0(X) \to \uAb$ taking an object $(x_0,x_1)\in X\times X$ to the abelian group $e_{x_0} A e_{x_1}$ and a multimorphism $((x_0,x_1),\ldots,(x_{n-1},x_n))\to (x_0,x_n)$ to the $n$-fold multiplication map $e_{x_0} A e_{x_1} \otimes \cdots \otimes e_{x_{n-1}} A e_{x_n} \to e_{x_0} A e_{x_n}$. 
\end{example}

\begin{remark}
We will mostly be concerned with multicategories enriched in groupoids. If $\uC$ and $\underline{\cal{D}}$ are two such multicategories, a \new{multifunctor of groupoid-enriched multicategories} from $\uC$ to $\underline{\cal{D}}$ must associate:
\begin{itemize}
    \item an object in $\underline{\cal{D}}$ to each object in $\uC$;
    \item a 1-multimorphism in $\underline{\cal{D}}$ to each 1-multimorphism in $\uC$;
    \item an (invertible) 2-morphism in $\underline{\cal{D}}$ to each (invertible) 2-morphism in $\uC$
\end{itemize}
such that:
\begin{itemize}
    \item multicomposition of 1-multimorphisms is strictly respected;
    \item identity 1-multimorphisms are strictly respected;
    \item vertical composition of 2-morphisms is strictly respected;
    \item identity 2-morphisms are strictly respected;
    \item horizontal multicomposition of 2-morphisms is strictly respected.
\end{itemize}
\end{remark}

\begin{definition}[cf. Section 2.4.1 of \cite{LLS-KST}]
Let $\mathscr{C}$ be a multicategory. The \new{canonical groupoid enrichment} $\widetilde{\mathscr{C}}$ of $\mathscr{C}$ is the groupoid-enriched multicategory defined as follows.
\begin{itemize}
    \item The objects of $\widetilde{\mathscr{C}}$ are the same as the objects of $\mathscr{C}$.
    \item For objects $x_1,\ldots,x_n$ and $y$ of $\mathscr{C}$, the groupoid of multimorphisms
    \[
        (x_1,\ldots,x_n) \to y
    \]
    in $\mathscr{C}$ has objects given by decorated rooted plane trees with $n$ leaves, allowing some leaves to be designated as ``stump leaves''. The tree must be equipped with a decoration of each edge by an object of $\mathscr{C}$ and each internal vertex (all vertices except the root and the non-stump leaves) by a multimorphism in $\mathscr{C}$ from the sequence of edges inputting to the vertex to the single edge outputting from the vertex. Each internal vertex that is a stump leaf with output edge labeled by $x \in \mathscr{C}$ must thus be labeled with a 0-input multimorphism from the empty sequence of inputs to $x$. The non-stump leaf edges are required to be decorated by $x_1,\ldots,x_n$ in order, and the root edge is required to be decorated by $y$. 
    \item For two decorated trees $T$ and $T'$ such that the multicomposition in $\mathscr{C}$ of the decorations of $T$ equals the multicomposition of the decorations of $T'$, the multimorphism groupoid from $(x_1,\ldots,x_n)$ to $y$ in $\widetilde{\mathscr{C}}$ is defined to have a unique morphism from $T$ to $T'$ which we will call the change-of-tree morphism from $T$ to $T'$. If the decorations of $T$ and $T'$ do not multicompose to the same multimorphism in $\mathscr{C}$, the set of 2-morphisms from $T$ to $T'$ is defined to be empty.
    \item Multicomposition of decorated trees $T_1,\ldots,T_n$ and $T'$, where the output objects of each $T_i$ agree with the input objects of $T'$, is defined by gluing $T_1,\ldots,T_n$ on top of $T'$ (the root vertices of the $T_i$ and the non-stump leaf vertices of $T'$ go away, rather than becoming internal vertices of the multicomposition), with decorations given by the decorations on internal vertices of $T_1, \ldots,T_n$ and $T'$.
    \item Vertical composition of change-of-tree morphisms $T \to T' \to T''$ is defined in the only possible way. Horizontal multicomposition of change-of-tree morphisms is also uniquely determined. 
\end{itemize}
\end{definition}

For a finite set $X$, we will let $\mathcal{S}(X)$ denote the canonical groupoid enrichment $\widetilde{\mathcal{S}^0(X)}$ of the shape multicategory of $X$.

We next consider the Burnside multicategory; in the definition we give, multicomposition is not strictly associative. This issue can be fixed in various ways, e.g. by replacing abstract finite sets with finite subsets of $\mathbb{R}^k$ for various $k$. We refer to \cite[Definition 2.56]{LLS-KST} for more details about this fix.

\begin{definition}[cf. Definition 2.56 and Section 3.2.1 of \cite{LLS-KST}]\label{def:BurnsideMulticat}
    Let $\underline{\mathscr{B}}$, the \new{Burnside multicategory}, denote the groupoid-enriched multicategory defined as follows.
    \begin{itemize}
        \item The objects of $\underline{\mathscr{B}}$ are finite sets $X$. 
        \item Given objects $X_1, \ldots, X_n$ and $Y$, a 1-multimorphism from $X_1,\ldots,X_n$ to $Y$ is a correspondence 
        \[
            Y \leftarrow A \rightarrow X_1 \times \cdots \times X_n
        \]
        of finite sets. One can think of $A$ as a matrix whose entries are finite sets; rows of the matrix are indexed by $Y$ and columns are indexed by $X_1 \times \cdots \times X_n$. This matrix interpretation is why we write $Y \leftarrow A \rightarrow (X_1 \times \cdots \times X_n)$ instead of $(X_1 \times \cdots \times X_n) \leftarrow A \rightarrow Y$.
        \item Let $n_1,\ldots,n_m \geq 0$, and suppose we have correspondences
        \[
            Y_i \leftarrow A_i \rightarrow X^i_1 \times \cdots \times X^i_{n_i}
        \]
        for $1 \leq i \leq m$ as well as
        \[
            Z \leftarrow B \rightarrow Y_1 \times \cdots \times Y_m.
        \]
        The multicomposition of these correspondences is the correspondence from $X^1_1 \times \cdots \times X^m_{n_m}$ to $Z$ obtained by, first, taking the product correspondence
        \begin{center}
        \[
            \xymatrix@C=-15mm{ & A_1 \times \cdots \times A_m \ar[dl] \ar[dr] \\  *+[l]{Y_1 \times \cdots \times Y_m} & & *+[r] {X^1_1 \times \cdots \times X^m_{n_m}}.}
        \]
        \end{center}
        In matrix terms, $A_1 \times \cdots \times A_m$ is a tensor-product matrix built from the $A_i$. One then matrix-multiplies $A_1 \times \cdots \times A_m$ with $B$ on the left; disjoint unions and Cartesian products of sets are used in place of algebraic addition and multiplication when matrix-multiplying. We will write $*$ for composition of correspondences, so that the above multicomposition is
        \[
            B * (A_1 \times \cdots \times A_m).
        \]
        \item For an object $X$ of $\underline{\mathscr{B}}$, the identity object $\mathrm{Id}_X$ of the 1-multimorphism groupoid from $X$ to $X$ is the correspondence
        \[
            X \leftarrow A \rightarrow X 
        \]
        where, for elements $x,y$ of $X$, the entry of $A$ in row $y$ and column $x$ is either empty if $x \neq y$ or a one-point set if $x = y$.
        \item For 1-multimorphisms $A,B$ from $X_1, \ldots, X_n$ to $Y$, a 2-morphism from $A$ to $B$ is an isomorphism of correspondences, i.e. an entrywise bijection from $A$ to $B$ viewed as set-valued matrices. Vertical composition of 2-morphisms is composition of entrywise bijections.
        \item Say $A_i$ ($1 \leq i \leq m$) and $B$ are correspondences whose multicomposition makes sense as above, and $A'_i$ and $B'$ are another such set of correspondences. Say we have 2-morphisms $f_i \colon A_i \to A'_i$ and $g \colon B \to B'$. The horizontal multicomposition $g * (f_1 \times \cdots \times f_m)$ of these 2-morphisms is given as follows.
        \begin{itemize}
            \item First take the product bijection
            \[
                f_1 \times \cdots \times f_m \colon (A_1 \times \cdots \times A_m) \to (A'_1 \times \cdots \times A'_m).
            \]
            \item Each matrix entry of $B * (A_1 \times \cdots \times A_m)$ is a union of Cartesian products of a matrix entry of $A_1 \times \cdots \times A_m$ with a matrix entry of $B$. Together, $f_1 \times \cdots \times f_m$ and $g$ induce maps on the Cartesian products; taking the union of all these maps gives a bijection from the entry of $B * (A_1 \times \cdots \times A_m)$ to the corresponding entry of $B' * (A'_1 \times \cdots \times A'_m)$. These bijections together define $g * (f_1 \times \cdots \times f_m)$.
        \end{itemize}
    \end{itemize}
\end{definition}

Adapting constructions of Sarkar, Scaduto, and Stoffregen \cite{SSS} to the multicategory setting, we also introduce a signed variant $\underline{\mathscr{B}}_{\sigma}$ of $\underline{\mathscr{B}}$.

\begin{definition}[cf. Section 3.2 of \cite{SSS}]
     Let $\underline{\mathscr{B}}_{\sigma}$, the \new{signed Burnside multicategory}, denote the groupoid-enriched multicategory defined as follows, phrased in terms of abstract sets but implicitly fixed to have strictly associative multicomposition as in \cite[Definition 2.56]{LLS-KST}.
    \begin{itemize}
        \item The objects of $\underline{\mathscr{B}}_{\sigma}$ are finite sets $X$.
        \item Given objects $X_1, \ldots, X_n$ and $Y$, a 1-multimorphism from $X_1, \ldots, X_n$ to $Y$ is a finite set $A$ equipped with both the data of a correspondence $Y \leftarrow A \rightarrow (X_1 \times \cdots \times X_n)$ and a map $\sigma_A \colon A \to \{+1,-1\}$. We think of $A$ as a set-valued matrix in which each element of each matrix entry of $A$ is labeled either plus or minus.
        \item Multicomposition of correspondences is defined as in Definition~\ref{def:BurnsideMulticat} with sign data given as follows: 
        \begin{itemize}
            \item The sign of an element of the Cartesian product $A_1 \times \cdots \times A_m$ is the product of the signs of its components.
            \item Any matrix entry in $B * (A_1 \times \cdots \times A_m)$ is a disjoint union, so an element $x$ of the matrix entry is an element of some term in the disjoint union, which is a pair of an element $x_B$ of an entry of $B$ and an element of an entry $x_A$ of $A_1 \times \cdots \times A_m$. The sign of $x$ is defined to be the sign of $x_A$ times the sign of $x_B$.
        \end{itemize}
        \item For an object $X$ of $\underline{\mathscr{B}}_{\sigma}$, the identity object $\mathrm{Id}_X$ of the 1-multimorphism groupoid from $X$ to $X$ is defined as in Definition~\ref{def:BurnsideMulticat} with each element of the correspondence $A$ labeled plus.
        \item 2-morphisms between 1-multimorphisms are entrywise bijections preserving the sign labeling. Vertical and horizontal composition of 2-morphisms are defined as in Definition~\ref{def:BurnsideMulticat}; note that if $f_1,\ldots,f_m$ and $g$ preserve sign labeling, then so does $g * (f_1 \times \cdots \times f_m)$.
    \end{itemize}
\end{definition}

\section{Frames as an alternate approach to the spectrification of \texorpdfstring{$H^n$}{Hn}}\label{sec:Frames}

\subsection{Shape multicategories}

First we review the shape multicategories used in \cite{LLS-KST}.

\subsubsection{Shape multicategory for an even number of points}

Let $\mathsf{B}_n$ denote the set of \new{crossingless matchings} on $2n$ points. A priori, elements of $\mathsf{B}_n$ are combinatorial, but we also fix a smooth embedding of each $a \in \mathsf{B}_n$ into $[0,1] \times (0,1)$ such that the endpoints of the arcs of $a$ are the points $[2n]_{\std} := \{(1,1/(2n+1)),\ldots,(1,2n/(2n+1))\}$. Furthermore, fix a small number $\varepsilon$ so that
\[
    a \cap ([1-\varepsilon,1] \times (0,1)) = [1-\varepsilon,1] \times [2n]_{\std}.
\]

\begin{definition}
    Let $\mathcal{S}^0_n = \mathcal{S}^0(\mathsf{B}_n)$ denote the \new{shape multicategory of $\mathsf{B}_n$} and let $\mathcal{S}_n = \mathcal{S}(\mathsf{B}_n)$ denote its canonical groupoid enrichment.
\end{definition}

Objects of either $\mathcal{S}^0_n$ or $\mathcal{S}_n$ are pairs $(a,b)$ where $a$ and $b$ are crossingless matchings on $2n$ points.

\subsubsection{Shape multicategory for flat tangles}

The following definition is an instance of a more general type (explained in \cite[Definition 2.3]{LLS-KST}) of shape multicategory  than in Definition~\ref{def:ShapeMulticat}.
\begin{definition}[cf. Section 3.2.2 of \cite{LLS-KST}]
    For $m, n \geq 0$, define $\mathcal{T}^0_{m,n}$ to be the multicategory with three kinds of objects: 
    \begin{enumerate}
        \item\label{it:BmPairs} pairs $(a,b)$ where $a,b \in \mathsf{B}_m$;
        \item\label{it:BnPairs} pairs $(a,b)$ where $a,b \in \mathsf{B}_n$;
        \item\label{it:Triples} triples $(a,T,b)$ where $a \in \mathsf{B}_m$ and $b \in \mathsf{B}_n$ ($T$ is just notation here).
    \end{enumerate}
    For objects $(a_0,a_1),\ldots,(a_{k-1},a_k)$ that are all of type \eqref{it:BmPairs}, there is a unique multimorphism
    \[
        (a_0,a_1),\ldots,(a_{k-1},a_k) \to (a_0,a_k)
    \]
    and no other multimorphisms with source $(a_0,a_1),\ldots,(a_{k-1},a_k)$. For objects that are all of type \eqref{it:BnPairs}, multimorphisms are defined similarly.

    Finally, whenever we have a sequence of objects
    \[
        (a_0,a_1), \ldots, (a_{k-1}, a_k), (a_k, T, b_0), (b_0,b_1),\ldots, (b_{l-1},b_l)
    \]
    where the first $k$ objects are of type \eqref{it:BmPairs} and the last $l$ objects are of type \eqref{it:BnPairs}, there is a unique multimorphism with this sequence of objects as its source. The target of this multimorphism is $(a_0,T,b_l)$. Let 
    \[
        {_{m}}\mathcal{T}_n = {_{m}^{\phantom{.}}\widetilde{\mathcal{T}}^0_n}
    \]
    be the canonical groupoid enrichment of ${_{m}^{\phantom{.}}}\mathcal{T}^0_n$.
\end{definition}

\subsubsection{Shape multicategory for general tangles}

Finally, there is a shape multicategory for non-flat tangles with $N$ crossings. As in \cite[Definition 3.6]{LLS-KST}, let $\underline{2}^N_0$ denote the category whose objects are elements $(v_1,\ldots,v_N) \in \{0,1\}^N$ and whose morphisms consist of a unique morphism from $(v_1,\ldots,v_N)$ to $(w_1,\ldots, w_N)$ whenever $v_i \leq w_i$ for $1 \leq i \leq N$.

\begin{definition}[cf. Section 3.2.4 of \cite{LLS-KST}]
    Let $(\underline{2}^N \widetilde{\times} {_m}\mathcal{T}_n)^0$ be the multicategory with objects of three types:
    \begin{enumerate}
        \item\label{it:NonFlatTangleMPair} Pairs $(a,b)$ with $a,b \in \mathsf{B}_m$;
        \item\label{it:NonFlatTangleNPair} Pairs $(a,b)$ with $a,b \in \mathsf{B}_n$;
        \item\label{it:NonFlatTangleQuadruple} Quadruples $(v,a,T,b)$ where $v \in \{0,1\}^N$ and $(a,T,b)$ is an object of ${_m}\mathcal{T}_n$
    \end{enumerate}
    and multimorphisms specified as follows. For $(a_0,a_1), \ldots, (a_{m-1},a_m)$ that are all of type \eqref{it:NonFlatTangleMPair}, there is a unique multimorphism from $((a_0,a_1),\ldots,(a_{m-1},a_m))$ to $(a_0,a_m)$. The same holds when we have objects that are all of type \eqref{it:NonFlatTangleNPair}. For a sequence of objects
    \[
        (a_0,a_1),\ldots,(a_{k-1},a_k), (v,a_k,T,b_0), (b_0,b_1), \ldots, (b_{l-1},b_l),
    \]
    and a morphism $v \to w$ in $\underline{2}^N$, there is a unique multimorphism from the given sequence to $(w,a_0,T,b_l)$.
\end{definition}

\begin{remark}
    In \cite{LLS-KST}, Lawson--Lipshitz--Sarkar do not take the canonical groupoid enrichment of $(\underline{2}^N \widetilde{\times} {_m}\mathcal{T}_n)^0$; rather, they let $\underline{2}^N \widetilde{\times} {_m}\mathcal{T}_n$ denote a smaller multicategory with fewer multimorphisms. To avoid confusion, we will use the notation $\widetilde{(\underline{2}^N \widetilde{\times} {_m}\mathcal{T}_n)^0}$ for the canonical groupoid enrichment of $(\underline{2}^N \widetilde{\times} {_m}\mathcal{T}_n)^0$.
\end{remark}

\subsection{Tangle diagrams}\label{sec:TangleDiagrams}

\begin{definition}[cf.\ Section 2.10 of \cite{LLS-KST}]
    Let $n,m \geq 0$. A \new{flat $(2m,2n)$-tangle diagram} $W$ is an embedded 1-dimensional cobordism in $[0,1] \times (0,1)$ from $\{0\} \times [2m]_{\std}$ at $\{0\} \times (0,1)$ to $\{1\} \times [2n]_{\std}$ at $\{1\} \times (0,1)$.
\end{definition}

For example, a crossingless matching $a \in \mathsf{B}_n$ is an example of a flat $(0,2n)$-tangle diagram. Flat tangle diagrams are projections to two dimensions of the $N=0$ case of the following definition for general (non-flat) tangles.

\begin{definition}[cf.\ Section 2.10 of \cite{LLS-KST}]
    A $(2m,2n)$-tangle diagram with $N$ crossings is an embedded 1-dimensional cobordism in $\mathbb{R} \times [0,1] \times (0,1)$ from $\{0\} \times \{0\} \times [2m]_{\std}$
    to $\{0\} \times \{1\} \times [2n]_{\std}$ whose projection to $[0,1] \times (0,1)$ has $N$ double points (called crossings) and no cusps, tangencies, or triple points, equipped with a total ordering of its crossings.
\end{definition}

\begin{remark}
    If we wanted, following \cite{LLS-KST} we could take equivalence classes of $(2m,2n)$-tangle diagrams modulo diffeomorphisms of the second factor of $\mathbb{R} \times [0,1] \times (0,1)$, and the below constructions would still work.
\end{remark}

Given a $(2m,2n)$-tangle diagram $T$ with $N$ (ordered) crossings, and an element $v$ of $\{0,1\}^N$, we can produce a flat $(2m,2n)$-tangle diagram $T_v$ by projecting $T$ to $[0,1] \times (0,1)$ and resolving crossings as in \cite[Figure 2.4]{LLS-KST} (the ordering of crossings tells us which entry of $v$ corresponds to which crossing). In slightly more detail, we could choose small disks around the crossings to define more precisely what it means to resolve a crossing, as in \cite[proof of Lemma 3.18]{LLS-KST}.

\subsection{Dotted cobordisms}

Unlike in \cite{LLS-KST}, we will be primarily concerned with cobordisms with corners, due to the difference in perspectives shown above in Figure~\ref{fig:NaturalMult}. 

\begin{definition}
    For $a,b \in \mathsf{B}_n$, a \new{dotted cobordism} with corners (or just dotted cobordism) from $a$ to $b$ is a compact orientable surface $\Sigma$ embedded in $[0,1] \times (0,1) \times [0,1]$, with boundary on 
    \[
        (\{0\} \times (0,1) \times [0,1]) \cup ([0,1] \times (0,1) \times \{1\}) \cup (\{1\} \times (0,1) \times [0,1]),
    \]
    such that:
    \begin{itemize}
        \item The boundary of $\Sigma$ on $\{0\} \times (0,1) \times [0,1]$ is $\mathrm{swap}(a)$ as an embedded 1-manifold, where $\mathrm{swap}(a)$ is the embedded submanifold of $(0,1) \times [0,1]$ obtained from $a \subset [0,1] \times (0,1)$ by swapping the coordinates;
        \item The boundary of $\Sigma$ on $[0,1] \times (0,1) \times \{1\}$ is $[0,1] \times [2n]_{\std} \times \{1\}$;
        \item The boundary of $\Sigma$ on $\{1\} \times (0,1) \times [0,1]$ is $\mathrm{swap}(b)$.
    \end{itemize}
    We also make the choice of some nonnegative integer (``number of dots'') for each connected component of $\Sigma$.
\end{definition}

\begin{remark}
    When we draw three-dimensional cubes like $[0,1] \times (0,1) \times [0,1]$, the first coordinate will be drawn left to right, the second coordinate will be drawn back to front, and the third coordinate will be drawn top to bottom. As a visual check, each of the three pieces on the bottom-left of Figure~\ref{fig:NaturalMult} should be a dotted cobordism, in the sense of the above definition, from the crossingless matching on its left side to the crossingless matching on its right side. Each of these crossingless matchings should be viewed as a one-dimensional cobordism from the empty set to a four-point set.
\end{remark}

If $\Sigma$ is a dotted cobordism from $a$ to $b$ and $\Sigma'$ is a dotted cobordism from $b$ to $c$, then we can shrink $\Sigma$ to live in $[0,1/2] \times (0,1) \times [0,1]$, shrink $\Sigma'$ to live in $[1/2,1] \times  (0,1) \times [0,1]$, and concatenate the results to get a dotted cobordism from $a$ to $c$. This operation descends to isotopy classes rel boundary.

\begin{remark}
    From the set of isotopy classes rel boundary of dotted cobordisms from $a$ to $b$, we can form the free abelian group $\text{Cob}(a,b)$ with basis given by these isotopy classes. We can then take the quotient of this free abelian group by the local Bar-Natan relations for dotted cobordisms. The concatenation operation induces $\mathbb{Z}$-bilinear maps from the Bar-Natan quotient for $\text{Cob}(a,b)$ times the Bar-Natan quotient for $\text{Cob}(b,c)$ to the Bar-Natan quotient for $\text{Cob}(a,c)$.
\end{remark}

\subsection{Basis elements from bounding disks}

\begin{definition} 
    For $a, b \in \mathsf{B}_n$, suppose that the subspace
    \[
        (\{0\} \times \mathrm{swap}(a)) \cup ([0,1] \times [2n]_{\std} \times \{1\}) \cup (\{1\} \times \mathrm{swap}(b))
    \]
    of $[0,1] \times (0,1) \times [0,1]$ consists of $k$ circles $C_1,\ldots,C_k$. We will say that the \new{number of circles joining $a$ and $b$} is $k$ and denote this count by $\#\textrm{Circ}(a,b)$.
\end{definition}

If $\#\textrm{Circ}(a,b) = k$, then for each circle $C_1,\ldots,C_k$, choose either ``dot'' or ``no dot.'' There is a unique, up to isotopy rel boundary, choice of a dotted cobordism from $a$ to $b$ consisting of $k$ embedded disks $D_1, \ldots, D_k$ bounding $C_1,\ldots, C_k$, with the number of dots on $D_i$ equal to zero if $C_i$ is labeled ``no dot'' and equal to one if $C_i$ is labeled ``dot.'' Different choices of dot patterns give us $2^k$ isotopy classes of dotted cobordisms. We will denote the set of these isotopy classes by $\mathrm{Disks}(a,b)$.

\begin{remark}
    The elements of $\mathrm{Disks}(a,b)$ give a $\mathbb{Z}$-basis for the Bar-Natan quotient group of $\text{Cob}(a,b)$.
\end{remark}

In connection with flat tangle diagrams, we will use a variant of $\mathrm{Disks}(a,b)$.

\begin{definition}
    For $a \in \mathsf{B}_m$, $b \in \mathsf{B}_n$, and a flat $(2m,2n)$-tangle diagram $T$, let $\overline{T}$ be $T$ viewed as a subspace of
    \[
       [0,1] \times (0,1) \times \{1\} \subset [0,1] \times (0,1) \times [0,1].
    \]
    Suppose that the subspace 
    \[
        (\{0\} \times \mathrm{swap}(a)) \cup \overline{T} \cup (\{1\} \times \mathrm{swap}(b))
    \]
    consists of $k$ circles $C_1,\ldots,C_k$. We will say that $\#\textrm{Circ}(a,T,b) = k$. In this case, there is a unique, up to isotopy rel boundary, choice of compact orientable surface $\Sigma$, with boundary on
    \[
        (\{0\} \times (0,1) \times [0,1]) \cup ([0,1] \times (0,1) \times \{1\}) \cup (\{1\} \times (0,1) \times [0,1])
    \]
    such that topologically $\Sigma$ consists of $k$ disks $D_1,\ldots,D_k$ bounding $C_1,\ldots,C_k$ respectively, and such that:
    \begin{itemize}
        \item The boundary of $\Sigma$ on $\{0\} \times (0,1) \times [0,1]$ is $\mathrm{swap}(a)$;
        \item The boundary of $\Sigma$ on $[0,1] \times (0,1) \times \{1\}$ is $\overline{T}$;
        \item The boundary of $\Sigma$ on $\{1\} \times (0,1) \times [0,1]$ is $\mathrm{swap}(b)$.
    \end{itemize}
    As before, for each circle $C_1,\ldots,C_k$, make a choice of ``dot'' or ``no dot.'' This labeling specifies an assignment of either zero dots or one dot on each $D_i$. We let $\mathrm{Disks}(a,T,b)$ denote the resulting set of isotopy classes of dotted embedded surfaces.
\end{definition}

\subsection{Frames}

We now discuss the frames that we mentioned above in the introduction.

\subsubsection{Frames for arc algebras}

\begin{definition}\label{def:Frames}
    Let $a_0,\ldots,a_m \in \mathsf{B}_n$. The \new{frame} $F_{a_0,\ldots,a_m}$ associated to $a_0,\ldots,a_m$ is the subset of $[0,1] \times (0,1) \times [0,1]$ defined as follows.
    \begin{itemize}
        \item Start with the union, from $i=0$ to $m$, of ``bridges'' 
        \[
            [2i/(2m+1), (2i+1)/(2m+1)] \times \mathrm{swap}(a_i)
        \]
        where we view $a_i$ as an embedded submanifold of $[0,1] \times (0,1)$.
        \item Take the additional union with ``rails''
        \[
            [0,1] \times [2n]_{\std} \times [1-\varepsilon,1].
        \]
    \end{itemize}
    See the left side of Figure~\ref{fig:BasicFrame} for an example of a frame. We think of a frame as embedded in a ``brick castle'' made from one slab at the bottom and $m+1$ bricks laid in parallel on top of it. We implicitly consider two frames to be the same if they are related by rescaling the height of the slab, height of the bricks, widths of the bricks, or widths of the gaps between bricks.
\end{definition}

\begin{figure}
    \centering
    \includegraphics[width=\textwidth]{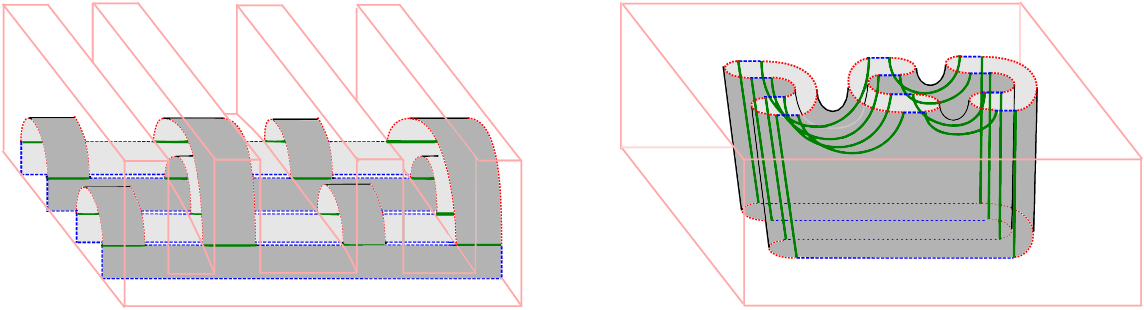}
    \caption{Left: a frame from Definition~\ref{def:Frames}, with its natural divided cobordism structure as in Remark~\ref{rem:DividedCob}. Right: the frame viewed as a multi-merge cobordism as in Remark~\ref{rem:FrameIsMultimerge}.  }
    \label{fig:BasicFrame}
\end{figure}

\begin{remark}\label{rem:FrameIsMultimerge}
    The frame $F_{a_0,\ldots,a_m}$ can be viewed as a canonical choice of multi-merge cobordism from $a_0 \overline{a_1} \sqcup \cdots \sqcup a_{m-1} \overline{a_m}$ to $a_0 \overline{a_m}$ as in \cite[Section 3.3]{LLS-KST}. See the right side of Figure~\ref{fig:BasicFrame}; we are using a homeomorphism between the brick castle and the cube in which the top faces of the leftmost and rightmost bricks get sent to the left and right faces of the cube. If we have dotted cobordisms from $a_{i-1}$ to $a_i$ for $1 \leq i \leq m$, their left-to-right composition can equivalently be obtained by gluing the dotted cobordisms into the internal boundary pieces of $F_{a_0,\ldots,a_m}$ as shown above in Figure~\ref{fig:NaturalMultUsingFrames}.
\end{remark}

\begin{remark}
    If we have $m$ sequences of elements of $H^n$ and we apply left-to-right composition to each sequence individually, then compose the result, we get the same thing as if we had lumped all the sequences together into a large sequence and had done the left-to-right composition all at once. This is one way to say that multiplication in $H^n$ is associative.

    In terms of frames, when doing the composition in two steps, one glues dotted cobordisms into $m$ separate frames first and then glues the results into the input slots of another connecting frame. By contrast, when doing the composition in one step, one glues all of the dotted cobordisms at the same time into the input slots of a single larger frame.

    Furthermore, when doing the two-step composition, it is equivalent to first glue the $m$ separate frames into the connecting frame, then add the dotted cobordisms in the resulting slots. This type of gluing of frames is shown in Figure~\ref{fig:FrameComposition} and corresponds to multicomposition of multimerge cobordisms. One can see from Figure~\ref{fig:FrameComposition} that after rescaling along the first dimension of $[0,1] \times (0,1) \times [0,1]$ as well as some rescalings involving the height $\varepsilon$ of the rails, the frame we assembled in this way agrees with the single larger frame we used to do the composition in one step. This is an equivalent frame-based way to say that multiplication in $H^n$ is associative.
\end{remark}

\begin{figure}
    \centering
    \includegraphics[width=\textwidth]{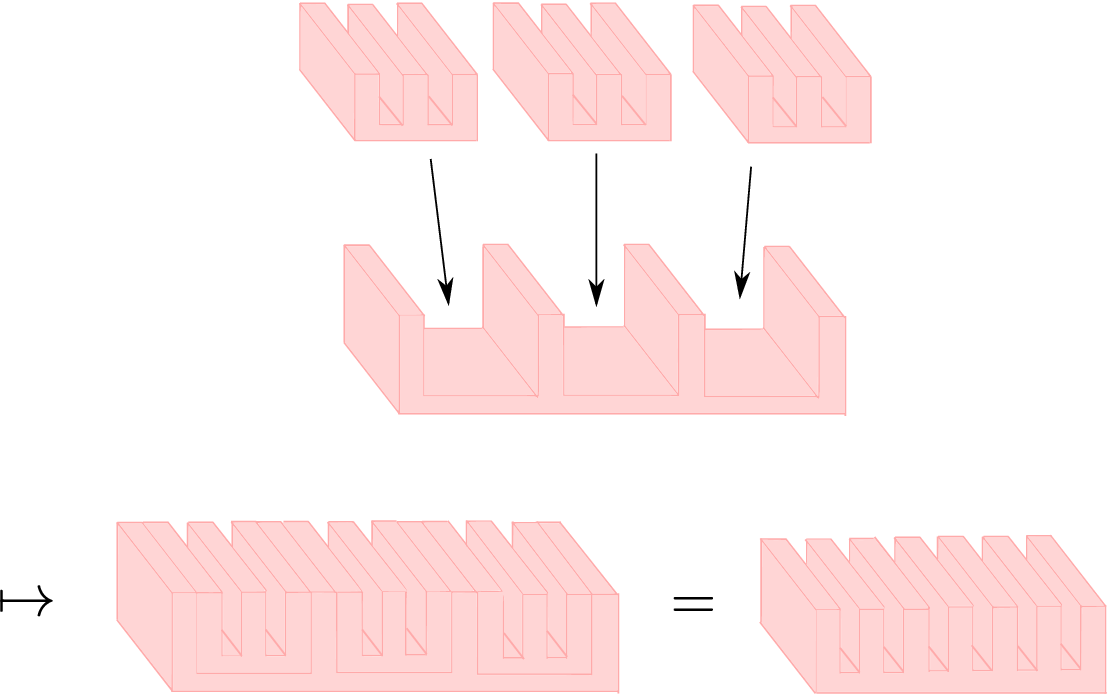}
    \caption{The gluing of frames that corresponds to multicomposition of multimerge cobordisms.}
    \label{fig:FrameComposition}
\end{figure}

\begin{remark}\label{rem:DividedCob}
    When interpreted as a multi-merge cobordism, the construction of $F_{a_0,\ldots,a_m}$ naturally gives it the structure of a divided cobordism as in \cite[Section 3.1]{LLS-KST}:
    \begin{itemize}
        \item $Z$ is the union of the internal boundary components of $F_{a_0,\ldots,a_m}$ and $Z'$ is the union of the external boundary components.
        \item $A$ and $A'$ are the intersections of $Z$ and $Z'$ with the sides of the bridges, so that $I$ and $I'$ are the intersections of $Z$ and $Z'$ with the rails.
        \item $\Gamma$ consists of the set of intervals along which the bridges are glued to the rails.
        \item When removing $\Gamma$, resulting in a disjoint union of bridges and rails, the bridges are the rectangles in item (I) of the definition of divided cobordism in \cite{LLS-KST}, and the rails are the higher polygons in item (II).
    \end{itemize}
    This structure is shown in Figure~\ref{fig:BasicFrame} using the visual conventions of \cite[Figure 3.2]{LLS-KST}. A general divided cobordism looks like a frame $F_{a_0,\ldots,a_m}$, except that the bridges are not required to be organized in neat sets of crossingless matchings along the first dimension of $[0,1] \times (0,1) \times [0,1]$. Our approach is to avoid divided cobordisms and their isotopies by working with frames instead when needed; the rescalings of frames we allowed are very mild in comparison with general isotopies.
\end{remark}

\subsubsection{Frames for tangle bimodules}

We will also use a flat-tangle variant of the frames $F_{a_0,\ldots,a_m}$.

\begin{definition}\label{def:FlatTangleFrame}
    Let $a_0,\ldots,a_k \in \mathsf{B}_m$ and $b_0,\ldots,b_l \in \mathsf{B}_n$. Let $T$ be a flat $(2m,2n)$-tangle diagram. The \new{flat tangle frame} $F_{a_0,\ldots,a_k,T,b_0,\ldots,b_l}$ associated to this data is the subset of $[0,1] \times (0,1) \times [0,1]$ defined as follows.
    \begin{itemize}
        \item Start with $F_{a_0,\ldots,a_m}$, scaled in the first dimension to live in
        \[
            [0,1/3] \times (0,1) \times [0,1].
        \]
        \item Take the union with $F_{b_0,\ldots,b_l}$, scaled in the first dimension to live in
        \[
            [2/3,1] \times (0,1) \times [0,1].
        \]
        \item Take the additional union with the set of points $(x,y,z) \in [1/3,2/3] \times (0,1) \times [0,1]$ such that $(x,y) \in T^{1/3,2/3}$ (the shrinking of $T$ along the first dimension to live in $[1/3, 2/3] \times (0,1)$) and $z \in [1 - \varepsilon,1]$ , where $\varepsilon$ is one common rail height chosen for both $F_{a_0,\ldots,a_k}$ and $F_{b_0,\ldots,b_l}$.
        \end{itemize}
\end{definition}

More generally, there is a variant of the frame construction for non-flat tangle diagrams.

\begin{definition}
    Let $a_0,\ldots,a_k \in \mathsf{B}_m$ and $b_0,\ldots,b_l \in \mathsf{B}_n$. Let $T$ be a $(2m,2n)$-tangle diagram with $N$ crossings, and let $v \to w$ be a morphism in $\underline{2}^N$. 
    The \new{general tangle frame} $F_{v \to w; a_0,\ldots,a_k,T,b_1,\ldots,b_l}$ associated to this data is defined the same way as $F_{a_0,\ldots,a_k,T,b_1,\ldots,b_l}$ except in the region $[1/3,2/3] \times (0,1) \times [0,1]$. In this region, rather than taking $(x,y,z) \in [1/3, 2/3] \times (0,1) \times [0,1]$ with $(x,y) \in T^{1/3,2/3}$ and $z \in [1 - \varepsilon,1]$, we take a standard cobordism from $T^{1/3,2/3}_v \times \{1-\varepsilon\}$ to $T^{1/3,2/3}_w \times \{1\}$ made up of standard saddles localized to the crossings being changed in going from $v$ to $w$. See Figure~\ref{fig:GeneralFrame}. 
\end{definition}

\begin{figure}
    \centering
    \includegraphics[width=\textwidth]{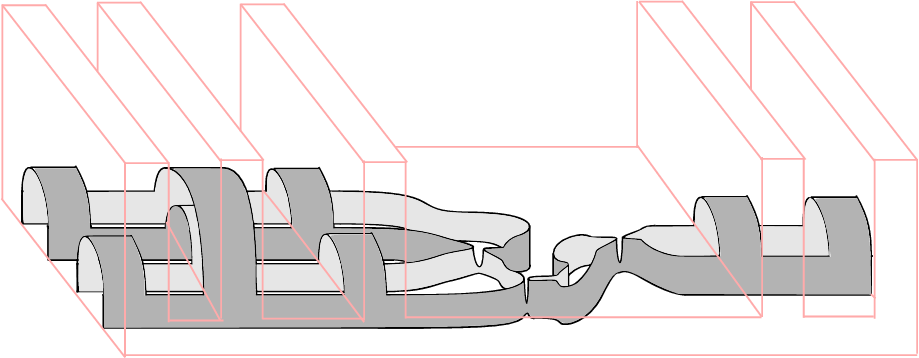}
    \caption{A general tangle frame $F_{v \to w; a_0,\ldots,a_k,T,b_1,\ldots,b_l}$.}
    \label{fig:GeneralFrame}
\end{figure}

\subsection{Mapping directly from the shape multicategory to the Burnside multicategory}

In this section we prove Theorem~\ref{thm:DirectlyToBurnside}, formulated using the more general shape multicategory $\widetilde{(\underline{2}^N \widetilde{\times} {_m}\mathcal{T}_n)^0}$. This version implies the corresponding versions with domain $\cal{S}_n$ or ${_{m}}\mathcal{T}_n$.

\begin{definition}\label{def:MultifunctorSl2Case}
    Define a multifunctor
    \[
        \Phi_n \colon \widetilde{(\underline{2}^N \widetilde{\times} {_m}\mathcal{T}_n)^0} \to \underline{\mathscr{B}}
    \]
    as follows.
    \begin{itemize}

        \item Objects $(a,b)$ with $a,b \in \mathsf{B}_m$ or $a,b \in \mathsf{B}_n$ get sent to $\mathrm{Disks}(a,b)$.

        \item Objects $(v,a,T,b)$ with $a \in \mathsf{B}_m$, $b \in \mathsf{B}_n$, and $v \in \underline{2}^N$ get sent to $\mathrm{Disks}(a,T_v,b)$ where, as specified in Section~\ref{sec:TangleDiagrams}, $T_v$ is the flat tangle diagram obtained from $T$ by resolving crossings according to $v$.
        
        \item A 1-multimorphism in in $\widetilde{(\underline{2}^N \widetilde{\times} {_m}\mathcal{T}_n)^0}$, either
        \[
            (a_0,a_1),\ldots,(a_{k-1},a_k) \to (a_0,a_k)
        \]
        (where $a_0,\ldots,a_k$ are all in $\mathsf{B}_m$ or all in $\mathsf{B}_n$) or
        \[
            (a_0,a_1),\ldots,(a_{k-1},a_k), (v,a_k,T,b_0), (b_0,b_1), \ldots, (b_{l-1},b_l) \to (w,a_0,T,b_l),
        \]
       is specified by a rooted plane tree $T$ with $k$ non-stump leaves. If $T$ has only one internal vertex, then the 1-multimorphism gets sent to a correspondence, either
        \[
            A \colon \mathrm{Disks}(a_0,a_1) \times \cdots \times \mathrm{Disks}(a_{k-1}, a_k) \to \mathrm{Disks}(a_0,a_k)
        \]
        or
        \begin{align*}
            &A \colon \mathrm{Disks}(a_0,a_1) \times \cdots \times \mathrm{Disks}(a_{k-1},a_k) \\
            &\times \mathrm{Disks}(a_k, T_v, b_0) \\
            &\times \mathrm{Disks}(b_0,b_1) \times \cdots \times \mathrm{Disks}(b_{l-1},b_l) \\
            &\to \mathrm{Disks}(a_0,T_w,b_l),
        \end{align*}
        that we now specify. Let $F$ be either $F_{a_0,\ldots,a_k}$ or $F_{v \to w; a_0,\ldots,a_k,T,b_0,\ldots,b_l}$ as appropriate.
        
        The frame $F$, as an abstract surface, may have various connected components, each of some genus. If $F$ has any components of genus $\geq 2$, every entry of the matrix $A$ is defined to be empty.
        
        Otherwise, $F$ may have some components of genus zero and some components of genus one. Say we have either $X_1 \in \mathrm{Disks}(a_0,a_1),\ldots, X_k \in \mathrm{Disks}(a_{k-1},a_k)$ or
        \begin{align*}
            &X_1 \in \mathrm{Disks}(a_0,a_1), \ldots, X_k \in \mathrm{Disks}(a_{k-1},a_k), \\
            &V \in \mathrm{Disks}(a_k,T,b_0), \\
            &Y_1 \in \mathrm{Disks}(b_0,b_1), \ldots, Y_l \in \mathrm{Disks}(b_{l-1},b_l)
        \end{align*}
        indexing a column of $A$ and $Z \in \mathrm{Disks}(a_0,a_k)$ indexing a row of $A$. Viewing $F$ as a multimerge cobordism, we obtain an abstract closed surface $\widehat{F}$ by gluing the dotted disks in $\{X_i\}$ or $\{X_i,V,Y_j\}$ to the internal boundary components of $F$ and the dot-reversal of the dotted disks in $Z$ (zero dots changed to one dot and vice-versa) to the outer boundary of $F$. 
        
        Connected components of $\widehat{F}$ naturally correspond to connected components of $F$, so for each component $\Sigma$ of $F$, we can look at the total number of dots on the corresponding component $\widehat{\Sigma}$ of $\widehat{F}$. If any genus-zero component does not have exactly one dot in total, or if any genus-one component does not have exactly zero dots in total, the entry of $A$ in the column and row in question is defined to be empty.

        In the remaining case, this entry of $A$ will be a Cartesian product of finite sets $S_{\Sigma}$ of size two, one for each genus-one component $\Sigma$ of $F$. Let $\Sigma$ be such a component, and let $\overline{\Sigma}$ denote the union of $\Sigma$ with the disks in $\{X_i\}$ or $\{X_i,V,Y_j\}$. Note that $\overline{\Sigma}$ is an embedded subset of $[0,1] \times (0,1) \times [0,1]$ whose dependence on the choices of representatives for the isotopy classes $\{X_i\}$ or $\{X_i,V,Y_j\}$ is relatively mild. 

        Consider the checkerboard coloring of the components of
        \[
            ([0,1] \times (0,1) \times [0,1]) \setminus \left( F \cup X_1 \cup \cdots \cup X_k \right)
         \]
         or
         \[
             ([0,1] \times (0,1) \times [0,1]) \setminus \left( F \cup X_1 \cup \cdots \cup X_k \cup V \cup Y_1 \cup \cdots \cup Y_l\right)
         \]
        in which $(x,y,z) \in [0,1] \times (0,1) \times [0,1]$ is colored white for $y$ very close to $0$ (or $1$, equivalently). By restricting the coloring to a small open neighborhood of $\overline{\Sigma}$ and extending the result globally, we get a checkerboard coloring of the components of
        \[
            ([0,1] \times (0,1) \times [0,1]) \setminus \overline{\Sigma}.
        \]
        Given the assumptions, the black region $B$ of this $\Sigma$-dependent checkerboard coloring will have first homology group isomorphic to $\mathbb{Z}$. In this case, we take $S_{\Sigma}$ to be the set consisting of the two possible generators for $H_1(B)$. The only choices we made were representatives for the isotopy classes $\{X_i\}$ or $\{X_i,V,Y_j\}$, and for any two such sets of choices there is an evident identification between the corresponding sets $S_{\Sigma}$. The matrix entry of $A$ in the row and column in question is now defined to be the product of the sets $S_{\Sigma}$ over the genus-one components $\Sigma$ of $F$.

        \item If the tree $T$ in the previous item has more than one internal vertex, each vertex of $T$ gets assigned a 1-multimorphism in $\underline{\mathscr{B}}$ by the previous item. Then $T$ uniquely specifies a combined 1-multimorphism obtained by repeated multicomposition in $\underline{\mathscr{B}}$ of the 1-multimorphisms for the vertices of $T$. 
        
        The entry of this multicomposition in column $\{X_i\}$ or $\{X_i,V,Y_j\}$ and row $Z$ can be described as the set of assignments of a label in $\{1,X\}$ to each boundary circle of the frame $F_{\nu}$ associated to each internal vertex ${\nu}$ of $T$ (if two of these circles are the same in $[0,1] \times (0,1) \times [0,1]$ then they must have the same label), agreeing with the subset of the labels that are determined by $\{X_i\}$ or $\{X_i,V,Y_j\}$ and $Z$, such that no $F_{\nu}$ has any genus $\geq 2$ components or components $\Sigma$ assigned the wrong number of dots by the collection of ``dot / no dot'' labels, together with a choice for every $\nu$ of an element of the labeling-determined entry of the local correspondence $A_{\nu}$ associated to $\nu$. These conditions imply that at each step of the process of gluing together the local frames $F_{\nu}$, we never produce a genus-zero or genus-one component with the wrong number of dots.
        
        \item For a change-of-tree 2-morphism in $\widetilde{(\underline{2}^N \widetilde{\times} {_m}\mathcal{T}_n)^0}$ from a tree $T$ to the corresponding basic tree $T_0$, the entrywise bijection $\phi_{T \to T_0}$ from the correspondence $A$ associated to $T$ to the correspondence $A_0$ associated to $T_0$ is defined as follows. Let 
        \[
            F = F_{a_0,\ldots,a_k}
        \]
        or
        \[
            F = F_{v \to w; a_0,\ldots,a_k,T,b_0,\ldots,b_l}
        \]
        be the frame associated to the basic tree $T_0$.
        
        If $F$ has any component $\Sigma$ of genus $\geq 2$, then every entry of $A_0$ is empty. For the entries of $A$, note that we can view $F$ as being glued along sets of circles from the local frames $F_{\nu}$ for the internal vertices $\nu$ of $T$. Let $\Sigma_{\nu}$ be the intersection of $F_{\nu}$ with the genus-two component $\Sigma$ of $F$.

        If any of the surfaces $\Sigma_{\nu}$ have genus $\geq 2$, then every entry of $A$ is also empty. Otherwise, when successively gluing the surfaces $\Sigma_{\nu}$ along sets of circles, at some point we either glue a genus-one component to a genus-one component, or we glue a genus-one component to a genus-zero component in a way that produces a genus-two component.

        Suppose that when successively gluing the $\Sigma_{\nu}$, at some point we glue two genus-one components together. For every nonempty entry of $A$, the total number of dots assigned by the corresponding ``dot / no dot'' labeling to each genus-one component in the gluing process must be zero, and all genus-one components must have all input circles labeled ``no dot'' and all output circles labeled ``dot.'' This is impossible when an input circle of some genus-one component is an output circle for some other genus-one component, so in this case the relevant entry of $A$ is empty.

        On the other hand, say that at some point in gluing together $\Sigma$, we glue together a genus-one and a genus-zero component and produce a genus-two component. Without loss of generality, the genus-one component is the input and the genus-zero component is the output. Since we produced a genus-two component, at least two output circles of the genus-one component agree with two input circles of the genus-zero component. To have a contribution to a nonempty entry of $A$, these two circles must be labeled $X$. But then the number of dots on the genus-zero component is at least two, a contradiction unless the relevant entry of $A$ is empty.

        In all these cases we take $\phi_{T \to T_0}$ to be the empty bijection between empty correspondences.
        
        Now suppose that all components of $F$ have genus either zero or one; the same is then true for the components of each $F_{\nu}$. Say we have $\{X_i\}$ or $\{X_i,V,Y_j\}$, and $Z$, such that the corresponding entry of $A_0$ is nonzero (if no ``dot / no dot'' labelings of the boundary circles of $F$ give a nonempty entry of $A_0$, then the same is true for $A$ whose entries involve more extensive choices of labelings than for $A_0$). 

        Elements of the entry of $A_0$ consist of choices, for each genus-one component $\Sigma$ of $F$, of one of the two possible generators for $H_1$ of the black region of the checkerboard coloring induced on $[0,1] \times (0,1) \times [0,1] \setminus \overline{\Sigma}$. We want a bijection between these elements and elements of the corresponding entry of $A$.

        Let $\Sigma$ be a component of $F$ and, for each internal vertex $\nu$ of $T$, let $\Sigma_{\nu} = \Sigma \cap F_{\nu}$. If $\Sigma$ has genus zero, then each $\Sigma_{\nu}$ has genus zero, and there is a unique valid assignment of ``dot / no dot'' labels to the boundary circles of all the surfaces $\Sigma_{\nu}$ extending the given labeling on boundary circles of $F$. Thus, to get an element of the entry of $A$ in question, we do not need to make choices involving $\Sigma$.

        Now assume $\Sigma$ has genus one, and furthermore there is some $\Sigma_{\nu}$ with genus one (all the rest must have genus zero). In this case, a generator for $H_1$ of the black region associated to $\Sigma_{\nu}$ canonically determines (by inclusion) a generator for $H_1$ of the black region associated to $\Sigma$. This choice of generator is the only choice we must make involving $\Sigma$ when producing an element of the entry of $A$; the labels on the boundary circles of all the $\Sigma_{\nu}$ are again uniquely determined in this case.

        Finally, suppose $\Sigma$ has genus one but all $\Sigma_{\nu}$ have genus zero. The given labeling on boundary circles of $F$ admits exactly two extensions to valid labelings of the boundary circles of each $\Sigma_{\nu}$, which can be described as follows. Form a graph whose vertices are the connected components of each $\Sigma_{\nu}$, with an edge from one connected component to another for each circle in their intersection. This graph has a unique minimal cycle up to cyclic reordering and direction reversal, involving a unique set of vertices $\Sigma_{\nu_1},\ldots,\Sigma_{\nu_{2q}}$ and unique edges between $\Sigma_{\nu_i}$ and $\Sigma_{\nu_{i+1}}$ for each $i \in \Z/2q\Z$. These unique edges each correspond to a circle and these circles must be labeled ``no dot, dot, no dot, dot, $\ldots$'' in alternating order as one traverses the cycle. There are two valid choices of labelings; given one choice, the other has ``dot'' swapped with ``no dot'' for every edge in the minimal cycle. To get an element of the entry of $A$ in question, the choice we need to make relating to $\Sigma$ is the choice of one of these two labelings.

        To connect with the choice we made relating to $\Sigma$ when defining an entry of $A_0$, we use a bijection between the above set of two labelings and the set of two possible generators for the black region of $[0,1] \times (0,1) \times [0,1] \setminus \overline{\Sigma}$. Given the labeling, look at any of the edges in the minimal cycle that are labeled ``no dot'' rather than ``dot'' (it does not matter which). If the pushoff $C_b$ of the corresponding circle $C \in \partial \Sigma_{\nu}$, oriented as the boundary of the corresponding 2d black region of the plane, into the 3d black region $B$ of $[0,1] \times (0,1) \times [0,1] \setminus \overline{\Sigma}$ is homologically nontrivial, then we take $[C_b]$ as the chosen generator of $H_1(B)$. 
        
        On the other hand, if $C_b$ is homotopically trivial, we choose a curve $D$ in $\Sigma$ satisfying $D \cdot C = +1$, where the orientation on $\Sigma$ is compatible with the orientation on $\overline{\Sigma}$ as the boundary of $B$. The pushoff $D_b$ of $D$ into $B$ is homologically nontrivial, and we take $[D_b]$ as the chosen generator of $H_1(B)$.
            
        \item For a general change-of-tree 2-morphism in $\widetilde{(\underline{2}^N \widetilde{\times} {_m}\mathcal{T}_n)^0}$ from a tree $T$ to a tree $T'$, we let
        \[
            \phi_{T \to T'} := \phi_{T' \to T_0}^{-1} \circ \phi_{T \to T_0}
        \]
        where $T_0$ is the basic tree associated to both $T$ and $T'$.
    \end{itemize}
\end{definition}

\begin{remark}
    For a tangle diagram $T$ with any choice of pox as in \cite[Definition 3.10]{LLS-KST}, Lawson--Lipshitz--Sarkar define a multifunctor $\underline{\mathsf{MB}}_T$ from the sub-multicategory 
    \[
        \underline{2}^N \widetilde{\times} {_m}\mathcal{T}_n \subset \widetilde{(\underline{2}^N \widetilde{\times} {_m}\mathcal{T}_n)^0}
    \]
    to $\underline{\mathscr{B}}$. By construction, $\underline{\mathsf{MB}}_T$ agrees with the restriction to $\underline{2}^N \widetilde{\times} {_m}\mathcal{T}_n$ of the multifunctor from Definition~\ref{def:MultifunctorSl2Case}. It follows from \cite[Theorem 1]{LLS-KST} that Definition~\ref{def:MultifunctorSl2Case} gives a valid multifunctor, naturally isomorphic to $\underline{\mathsf{MB}}_T$; we give an adaptation of the proof below to show explicitly that it can be formulated without the use of Lawson--Lipshitz--Sarkar's divided cobordism category.
\end{remark}

\begin{theorem}\label{thm:Sl2MultifunctorValid}
    Definition~\ref{def:MultifunctorSl2Case} defines a valid multifunctor of groupoid-enriched multicategories from $\widetilde{(\underline{2}^N \widetilde{\times} {_m}\mathcal{T}_n)^0}$ to $\underline{\mathscr{B}}$.
\end{theorem}

\begin{proof}
    First note that we have a functor on each multimorphism groupoid in $\widetilde{(\underline{2}^N \widetilde{\times} {_m}\mathcal{T}_n)^0}$; in other words, the vertical composition 
    \[
        (T \to T'') = (T' \to T'') \circ (T \to T')
    \]
    of change-of-tree 2-morphisms in $\widetilde{(\underline{2}^N \widetilde{\times} {_m}\mathcal{T}_n)^0}$ (where all of $T,T',T''$ have the same underlying basic tree $T_0$) is sent to
    \begin{align*}
        \phi_{T \to T''}    &= \phi_{T'' \to T_0}^{-1} \circ \phi_{T \to T_0} \\
                            &= \phi_{T'' \to T_0}^{-1} \circ \phi_{T' \to T_0} \circ \phi_{T'\to T_0}^{-1} \circ \phi_{T \to T_0} \\
                            &= \phi_{T' \to T''} \circ \phi_{T \to T'},
    \end{align*}
    and an identity change-of-tree 2-morphism $T \to T$ in $\widetilde{(\underline{2}^N \widetilde{\times} {_m}\mathcal{T}_n)^0}$ is sent to
    \[
        \phi_{T \to T} = \phi_{T \to T_0}^{-1} \circ \phi_{T \to T_0} = \id,
    \]
    the identity element in the relevant multimorphism groupoid of $\underline{\mathscr{B}}$.

    Next, for an object of $\widetilde{(\underline{2}^N \widetilde{\times} {_m}\mathcal{T}_n)^0}$ of the form $(a,b)$ where $a,b \in \mathsf{B}_n$ or $a,b \in \mathsf{B}_m$, we look at where the identity 1-multimorphism from $(a,b)$ to itself (given by the single-edge tree) gets sent. Say it gets sent to the correspondence $A$ from $\mathrm{Disks}(a,b)$ to itself. The frame $F_{a,b}$ is a disjoint union of $\#\textrm{Circ}(a,b)$ cylinders, and for $X,Y \in \mathrm{Disks}(a,b)$, if $X \neq Y$ then for at least one of the cylinders the dot patterns will make the matrix entry of $A$ in row $Y$ and column $X$ empty. On the other hand, if $X = Y$, there will be no choices involved in defining this matrix entry of $A$, which will be a one-point set. The result is the identity multimorphism in $\underline{\mathscr{B}}$ from $\mathrm{Disks}(a,b)$ to itself. The above argument holds equally well for an object of the form $(v,a,T,b)$ by replacing $\mathrm{Disks}(a,b)$ with $\mathrm{Disks}(a,T_v,b)$, $F_{a,b}$ with $F_{v \to v; a, T_v, b}$, and $\#\textrm{Circ}(a,b)$ with $\#\textrm{Circ}(a,T_v,b)$.

    Say that we horizontally compose some 1-multimorphisms in $\widetilde{(\underline{2}^N \widetilde{\times} {_m}\mathcal{T}_n)^0}$, which are given by trees $T_1,\ldots,T_m$, and then take the correspondence $A$ associated to the composite tree $T$. Since $T$ is a composite tree, $A$ is defined to be a multicomposition in $\underline{\mathscr{B}}$ of the correspondences for the basic-tree pieces of $T_1, \ldots, T_m$ taken all together. On the other hand, the correspondence associated to $T_i$ is the multicomposition of correspondences for just the basic-tree pieces of $T_i$. Multicomposing each set of correspondences for basic pieces to form the correspondences of $T_1, \ldots, T_m$, and then multicomposing the results, is the same as multicomposing all the basic-piece correspondences together in one step, because multicomposition of 1-multimorphisms in $\underline{\mathscr{B}}$ is strictly associative. Thus, the multifunctor strictly respects multicomposition of 1-morphisms.

    Finally, say that we horizontally compose some 2-morphisms in $\widetilde{(\underline{2}^N \widetilde{\times} {_m}\mathcal{T}_n)^0}$, from trees $T_i$ to $\widetilde{T}_i$ for $1 \leq i \leq m$ and from $T'$ to $\widetilde{T}'$ where $T', \widetilde{T}'$ each have $m$ inputs, and then take the entrywise bijection of correspondences
    \[
        \phi_{T' * (T_1, \ldots, T_m) \to \widetilde{T}' * (\widetilde{T}_1, \ldots, \widetilde{T}_m)}
    \]
    between the correspondences associated to the composite trees. Let $T_0$ be the basic tree associated to both $T' * (T_1,\ldots,T_m)$ and $\widetilde{T}' * (\widetilde{T}_1, \ldots, \widetilde{T}_m)$; we have
    \[
        \phi_{T' * (T_1, \ldots, T_m) \to \widetilde{T}' * (\widetilde{T}_1, \ldots, \widetilde{T}_m)} = \phi_{\widetilde{T}' * (\widetilde{T}_1, \ldots, \widetilde{T}_m) \to T_0}^{-1} \circ \phi_{T' * (T_1, \ldots, T_m) \to T_0}
    \]
    by definition.  
    
    We claim that $\phi_{T'*(T_1,\ldots,T_m) \to T_0}$ can be factored as
    \begin{equation}\label{eq:PhiFactorization}
        \begin{aligned}
            T' * (T_1,\ldots,T_m) & \xrightarrow{\left( \phi_{T' \to T'_0} * (\phi_{T_1 \to T_{1,0}} \times \cdots \times \phi_{T_m \to T_{m,0}} ) \right)} T'_0 * (T_{1,0},\ldots,T_{m,0}) \\
        & \xrightarrow{\phi_{T'_0 * (T_{1,0},\ldots,T_{m,0}) \to T_0}} T_0,
        \end{aligned}
    \end{equation}
    where $T'_0$ (respectively, $T_{i,0}$) are the basic trees associated to $T'$ and $\widetilde{T}'$ (respectively, $T_i$ and $\widetilde{T}_i$). 
    
    Indeed, we can assume that the frame $F$ of $T_0$ has no components of genus $\geq 2$ and that we are looking at a nonempty entry of the correspondence $A$ for $T' * (T_1,\ldots,T_m)$. Let $\Sigma$ be a genus-one component of $F$; we can write $\Sigma = \Sigma' \cup \Sigma_1 \cup \cdots \cup \Sigma_m$ where $\Sigma'$ is the portion of $\Sigma$ local to the tree $T'$ and $\Sigma_i$ is the portion of $\Sigma$ local to the tree $T_i$. The surface $\Sigma'$ is further cut into pieces along sets of circles according to the internal vertices of $T'$, and $\Sigma_i$ is cut into pieces along sets of circles according to the internal vertices of $T_i$. Let $\{\Sigma_{\nu}\}$ denote the set of all these minimal pieces of $\Sigma'$ and the $\Sigma_i$.

    If any $\Sigma_{\nu}$ has a genus-one component, then an element $x$ of the entry of $A$ in question already has its $\Sigma$-choice encoded by a generator of $H_1$ of the black region $B_{\nu}$ associated to $\Sigma_{\nu}$. The map $\phi_{T' * (T_1,\ldots,T_m) \to T_0}$ applied to $x$ gives the entry of $A_0$ whose $\Sigma$-choice is the image of this generator of $H_1(B_{\nu})$ under the inclusion map from $B_{\nu}$ into the black region $B$ associated to $\Sigma$. The composite map of \eqref{eq:PhiFactorization} gives the same entry of $A_0$; it passes an entry of $H_1(B_{\nu})$ through two inclusion maps in turn instead of their composition all at once, but these things are the same, so the factorization of \eqref{eq:PhiFactorization} holds when applied to $x$.

    Next, say all the $\Sigma_{\nu}$ have only genus-zero components, but one of $\{\Sigma',\Sigma_1,\ldots,\Sigma_m\}$ has a genus-one component. An element $x$ of the entry of $A$ in question has its $\Sigma$-choice specified by one of two possible ``no dot, dot, no dot, dot, $\ldots$'' labelings of a minimal cycle in a graph. Applying $\phi_{T' * (T_1,\ldots,T_m) \to T_0}$, we translate this $\Sigma$-choice to a generator of $H_1$ for the black region associated to $\Sigma$. By comparison, when we apply the composite map of \eqref{eq:PhiFactorization} to $x$, the first factor already translates the $\Sigma$-choice of $x$ to a $H_1$ generator of the black region (which is already visible when restricting attention to one of $\{\Sigma',\Sigma_1,\ldots,\Sigma_m\}$), and the second factor applies a transparent inclusion map. Thus, the factorization of \eqref{eq:PhiFactorization} holds when applied to $x$.

    Finally, say all the $\Sigma_{\nu}$ and all of $\{\Sigma',\Sigma_1,\ldots,\Sigma_m\}$ have only genus-zero components. An element $x$ of the entry of $A$ in question again has its $\Sigma$-choice specified by a ``no dot, dot, no dot, dot, $\ldots$'' labeling. When we apply the composite map of \eqref{eq:PhiFactorization} to $x$, the first factor coarse-grains this labeling by restricting it to circles on the boundary of $\Sigma'$ or some $\Sigma_i$. The second factor then translates this coarse-grained labeling into a $H_1$ generator for the black region. By comparison, the map $\phi_{T' * (T_1,\ldots,T_m) \to T_0}$ does the translation all at once, without coarse-graining first. Doing the coarse-graining first does not affect the result, so the factorization of \eqref{eq:PhiFactorization} holds when applied to $x$ and thus in general, proving the claim.

    To complete the proof of the theorem, say that instead of horizontally composing the 2-morphisms in $\widetilde{(\underline{2}^N \widetilde{\times} {_m}\mathcal{T}_n)^0}$ first, and then taking the associated entrywise bijection of correspondences, we take the individual entrywise bijections
    \begin{equation}\label{eq:IndividualEntrywiseBijections}
        \phi_{T' \to \widetilde{T'}}, \quad \phi_{T_1 \to \widetilde{T}_1},  \quad \ldots, \quad \phi_{T_m \to \widetilde{T}_m}
    \end{equation}
    and horizontally multicompose them in $\underline{\mathscr{B}}$. By definition,
    \[
        \phi_{T' \to \widetilde{T'}} = \phi_{\widetilde{T}' \to T'_0}^{-1} \circ \phi_{T' \to T'_0}
    \]
    and
    \[
        \phi_{T_i \to \widetilde{T_i}} = \phi_{\widetilde{T}_i \to T_{i,0}}^{-1} \circ \phi_{T_i \to T_{i,0}}.
    \]
     Since the multicomposition maps in a groupoid-enriched multicategory are functors of groupoids, when we horizontally multicompose the set of vertical compositions in \eqref{eq:IndividualEntrywiseBijections}, we can equivalently write the result as
    \[
        (\phi_{\widetilde{T}' \to T'_0} * (\phi_{\widetilde{T}_1 \to T_{1,0}} \times \cdots \times \phi_{\widetilde{T}_m \to T_{m,0}}))^{-1} \circ (\phi_{T' \to T'_0} * (\phi_{T_1 \to T_{1,0}} \times \cdots \times \phi_{T_m \to T_{m,0}} )),
    \]
    or redundantly as
    \begin{align*}
        & (\phi_{\widetilde{T}' \to T'_0} * (\phi_{\widetilde{T}_1 \to T_{1,0}} \times \cdots \times \phi_{\widetilde{T}_m \to T_{m,0}}))^{-1}\\
        & \circ \phi_{T'_0 * (T_{1,0},\ldots,T_{m,0}) \to T_0}^{-1} \\
        & \circ \phi_{T'_0 * (T_{1,0},\ldots,T_{m,0}) \to T_0} \\
        & \circ (\phi_{T' \to T'_0} * (\phi_{T_1 \to T_{1,0}} \times \cdots \times \phi_{T_m \to T_{m,0}} )).
    \end{align*}
    By equation~\eqref{eq:PhiFactorization} and its $\widetilde{T}$ analogue, we get
    \[
        \phi_{\widetilde{T}' * (\widetilde{T}_1, \ldots, \widetilde{T}_m) \to T_0}^{-1} \circ \phi_{T' * (T_1, \ldots, T_m) \to T_0}
    \]
    which equals $\phi_{T'*(T_1,\ldots,T_m) \to \widetilde{T}' * (\widetilde{T}_1,\ldots, \widetilde{T}_m)}$. It follows that our multifunctor strictly respects horizontal composition of 2-morphisms, so it is a valid multifunctor of groupoid-enriched multicategories.
\end{proof}

\begin{remark}
    We went into detail about the proof that the multifunctor of Definition~\ref{def:MultifunctorSl2Case} respects horizontal composition of 2-morphisms because that property seems to be the main difficulty in defining such a multifunctor. If we did not require this property, for example, we could have defined the entrywise bijections $\phi_{T \to T_0}$ arbitrarily instead of via the canonical procedure of Definition~\ref{def:MultifunctorSl2Case}, and we could also have defined the matrix entries of the correspondences associated to basic trees to be arbitrary finite sets of the right cardinality.
\end{remark}

\section{Blanchet--Khovanov algebras and the categorified quantum group}\label{sec:BlanchetKhovanovAlgebras}

\subsection{Webs and foams}

Following \cite{EST1}, let $\bbl$ denote the set of finite sequences of elements of $\{0,1,2\}$. Given a finite sequence $\vec{k}\in\bbl$, say of length $n$, 
we will view the entries of $\vec{k}$ as labeling the elements of $[n]_{\std} := (1/(n+1),2/(n+1),...,n/(n+1))$.

\begin{definition}[cf.\ Section~2.1 of \cite{EST1}]\label{def:UpwardWeb}
    Given $\vec k_1,\vec k_2\in\bbl$ (say of lengths $n_1$ and $n_2$), an \new{upward-pointing $\mathfrak{gl}_2$-web} (or just \new{$\mathfrak{gl}_2$-web}) $w$ from $\vec{k}_1$ to $\vec{k}_2$ is an oriented trivalent graph embedded in $[0,1] \times (0,1)$ with boundary the subset of $(\{0\} \times [n_1]_{\std}) \cup (\{1\} \times [n_2]_{\std})$ whose corresponding entries of $\vec{k}_i$ are nonzero, equipped with an orientation of each edge such that the first coordinate of $[0,1] \times (0,1)$ is strictly increasing in the direction of the given orientation, and equipped with a labeling of edges as either ``thickness one'' or ``thickness two'' such that each vertex has either two thickness one edges incoming and one thickness two edge outgoing, or the same with incoming and outgoing reversed. We will sometimes write $w \colon \vec{k}_1 \to \vec{k}_2$.
\end{definition}

Locally, $\gl_2$-webs look one of the pieces shown in Figure~\ref{fig:LocalGl2Webs} if we draw the first coordinate bottom to top and the second coordinate left to right (this is one of various standard conventions). All edges are oriented upwards, although we have drawn the orientations only on the thickness-one edges.

\begin{figure}[ht!]
    \centering
    \includegraphics{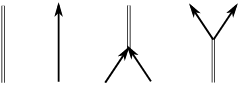}
    \caption{Local models for $\gl_2$-webs.}
    \label{fig:LocalGl2Webs}
\end{figure}

\begin{definition}[cf.\ Section~4.1 of \cite{EST1}]
    If $w \colon \vec{k}_1 \to \vec{k}_2$ is a $\mathfrak{gl}_2$-web, the \new{underlying topological web} $\hat{w}$ of $w$ is the flat tangle obtained by removing all thickness-two edges from $w$.
\end{definition}

The underlying topological webs of the local $\gl_2$-webs in Figure~\ref{fig:LocalGl2Webs} look like
\begin{figure}[ht!]
    \includegraphics{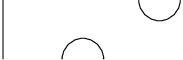}
    \caption{Underlying topological webs for local models for $\gl_2$-webs.}
    \label{fig:LocalTopWebs}
\end{figure}
when drawn with first coordinate bottom to top and second coordinate left to right. We do not consider orientations on underlying topological webs.

For $n \geq m \geq 0$, let $\omega_{n,m} = (1,\ldots,1, 0, \ldots, 0)$ with $n$ entries, the first $m$ of which are $1$. 

\begin{definition}[cf. Definition 2.5 of \cite{EST1}]
    A \new{$\mathfrak{gl}_2$-foam} $F$ (analogous to the pre-foams of \cite{EST1,lau-skew}) is a subset of $[0,1] \times (0,1) \times [0,1]$ that is locally an embedding of a Y shape times an interval, equipped with a labeling of facets as either ``thickness one'' or ``thickness two'' such that in a neighborhood of each singular seam, two legs of the Y shape have thickness one and the other has thickness two, and also equipped with an orientation for each singular seam, such that the intersection of $F$ with each of the four boundary components of $[0,1] \times (0,1) \times [0,1]$ is a $\mathfrak{gl}_2$-web (oriented top to bottom on the left and right faces $\{0,1\} \times (0,1) \times [0,1]$ and oriented left to right on the top and bottom faces $[0,1] \times (0,1) \times \{0,1\}$). Finally, each facet of $F$ should carry some number of dots.
\end{definition}

\begin{remark}
    Our visual conventions imply that ``upward-pointing'' webs are actually drawn pointing downwards when they appear on the left and right sides of the cube $[0,1] \times (0,1) \times [0,1]$.
\end{remark}

In particular, for $i = 1,2$, if $\vec{k}_i \in \bbl$ and $w_i$ are $\mathfrak{gl}_2$-webs from $\vec{k}_1$ to $\vec{k}_2$, a $\mathfrak{gl}_2$-foam $F$ from $w_1$ to $w_2$ is defined to be a $\mathfrak{gl}_2$-foam having the following boundary intersections:
\begin{itemize}
    \item $F \cap (\{0\} \times (0,1) \times [0,1]) = w_1$ (on the left face of the cube pointing downwards);
    \item $F \cap (\{1\} \times (0,1) \times [0,1]) = w_2$ (on the right face of the cube pointing downwards);
    \item $F \cap ([0,1] \times (0,1) \times \{0\})$ is the identity web from $\vec{k}_1$ to itself (on the top face of the cube pointing to the right);
    \item $F \cap ([0,1] \times (0,1) \times \{1\})$ is the identity web from $\vec{k}_2$ to itself (on the bottom face of the cube pointing to the right).
\end{itemize}

\begin{definition}
    If $F$ is a $\mathfrak{gl}_2$-foam, its \new{underlying topological foam} is the dotted cobordism obtained by removing all thickness-two facets, including their boundary intersections, from $F$.
\end{definition}

\subsection{The foam based Blanchet--Khovanov algebras}

Note that an element $\vec{k} \in \bbl$ has an even number of ones if and only if the entries of $\vec{k}$ sum to an even number. Following \cite{EST1}, we will refer to such $\vec{k}$ as ``balanced.''

\begin{definition}\label{def:rlFg}
    Let $\vec{k}$ be a balanced length-$n$ sequence of elements of $\{0,1,2\}$; say $2m$ is the sum of the entries of $\vec{k}$. Define $W_{\vec{k}}$ to be, in the language of \cite[Lemma 4.8]{EST1}, the set of the unique (up to isotopy and we pick one representative per class) \new{$F$-generated webs obtained by preferring right to left}
    \[
        u \colon 2\omega_{n,m} \to \vec{k}
    \]
    whose underlying topological web $\hat{u}$ is a crossingless matching on $2p$ points, where $p$ is the number of entries of $\vec{k}$ that are equal to one. Two examples of such webs $u$ are shown in \cite[equation (37)]{EST1}; another is shown on the right of Figure~\ref{fig:FGeneratedWeb}, with first coordinate bottom to top and second coordinate left to right.
\end{definition}

\begin{figure}
    \centering
    \includegraphics{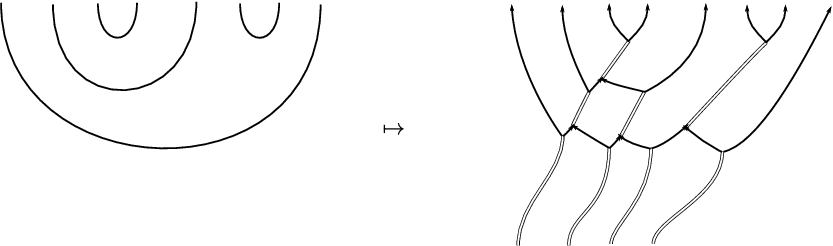}
    \caption{The unique $F$-generated and rightward-preferring web associated to a crossingless matching as in Definition~\ref{def:rlFg}. To construct the web, start with a set of parallel vertical 2-labeled strands below the crossingless matching. Successively take the rightmost available 2-labeled strand at the bottom and connect it to the lowest point of whichever still-available matching arc has rightmost right endpoint. If it is necessary to cross other matching arcs when doing this, always emerge a bit to the right of where you entered (going bottom to top).}
    \label{fig:FGeneratedWeb}
\end{figure}

For any balanced $\vec{k} \in \bbl$, suppose that $\vec{k}$ has an even number of ones, say $2p$ (so that $2m$ is $2p$ plus twice the number of twos in $\vec{k}$). Given a crossingless matching $a \in \mathsf{B}_p$, \cite[Lemma 4.8]{EST1} implies that there is a unique element of $W_{\vec{k}}$  whose underlying topological web is $a$; see Figure~\ref{fig:FGeneratedWeb}. We thus have a bijection between $\mathsf{B}_p$ and $W_{\vec{k}}$.

\begin{remark}
    The important part for us is that given $\vec{k}$ with $2p$ ones and entries summing to $2m$, for each crossingless matching $a \in \mathsf{B}_p$ we choose one web from $2\omega_{n,m}$ to $\vec{k}$ whose underlying topological web is $a$. The $F$-generated and right-preferring webs of \cite{EST1} are one way to do this, but any choice will work equally well for us.
\end{remark}

\begin{definition}[cf.\ Definitions 2.17, 2.19 and 3.11 and equation (40) of \cite{EST1}]\label{def:BKAlg}
    Given $\vec{k} \in \bbl$, let $w_i \in W_{\vec{k}}$ for $i = 1,2$. Define ${_{w_1}}(\mathfrak{W}^{\natural}_{\vec{k}})_{w_2}$ to be the free abelian group with basis formally given by isotopy classes rel boundary of $\mathfrak{gl}_2$-foams from $w_1$ to $w_2$, modulo the local $\mathfrak{gl}_2$-foam relations described in \cite[(8)--(10) and Lemmas 2.12--2.14]{EST1}. Define a multiplication map
    \[
        {_{w_1}}(\mathfrak{W}^{\natural}_{\vec{k}})_{w_2}
        \underset{\mathbb Z}{\otimes}
        {_{w_2}}(\mathfrak{W}^{\natural}_{\vec{k}})_{w_3} \to {_{w_1}}(\mathfrak{W}^{\natural}_{\vec{k}})_{w_3}
    \]
    sending $F \otimes F'$ to the concatenation of $F$ (shrunk to live in $[0,1/2] \times (0,1) \times [0,1]$) with $F'$ (shrunk to live in $[1/2,1] \times (0,1) \times [0,1]$) along $w_2$. By the \new{Blanchet--Khovanov algebra} associated to $\vec{k}$ we will mean the algebra
    \[
        \mathfrak{W}^{\circ}_{\vec{k}} := \bigoplus_{u,v \in W_{\vec{k}}} {_{u}}(\mathfrak{W}^{\natural}_{\vec{k}})_{v};
    \]
    by \cite[Lemma 2.27, Corollary 4.17, and Theorem 4.18]{EST1}, this definition is equivalent to Ehrig--Stroppel--Tubbenhauer's.
\end{definition}

\subsection{The \texorpdfstring{$\mathfrak{gl}_2$-foam}{gl(2)-foam} 2-category}

There is a foam 2-category $\mathfrak{F}$ generalizing Definition~\ref{def:BKAlg}.

\begin{definition}[cf. Definition 2.17 of \cite{EST1}]
     The $\mathfrak{gl}_2$-foam 2-category $\mathfrak{F}$ has set of objects $\bbl$, together with a zero object (Ehrig--Stroppel--Tubbenhauer do not include a zero object, but it will be necessary for the 2-functors we consider). For any $\vec{k}_1, \vec{k}_2$ in $\bbl$, there is a $\mathbb{Z}$-linear category $\mathrm{Hom}_{\mathfrak{F}}(\vec{k}_1, \vec{k_2})$ whose objects are $\mathfrak{gl}_2$-webs $w$ from $\vec{k}_1$ to $\vec{k}_2$ and whose abelian group $2\mathrm{Hom}_{\mathfrak{F}}(w_1,w_2)$ of morphisms from one such web $w_1$ to another $w_2$ is the free abelian group with basis formally given by isotopy classes rel boundary of $\mathfrak{gl}_2$-foams from $w_1$ to $w_2$ modulo the local $\mathfrak{gl}_2$-foam relations described in \cite[(8)--(10) and Lemmas 2.12--2.14]{EST1}. Morphism categories involving the zero object are defined to be zero.
\end{definition}

\begin{remark}
    For $\vec{k} \in \bbl$, the Blanchet--Khovanov algebra associated to $\vec{k}$ is the algebra obtained from the category $\mathrm{Hom}_{\mathfrak{F}}(2\omega_{n,m}, \vec{k})$ by taking the direct sum of all morphism spaces between objects $u,v \in W_{\vec{k}}$.
\end{remark}

Relatedly, in Queffelec--Rose \cite[Definitions 3.1 and 3.5]{QR} there is a definition of a foam 2-category $m\mathrm{Foam}_n(N)$ for general $n,m,N$ (we have swapped the roles of $m$ and $n$ in their indexing).  

\begin{proposition}
    There is an isomorphism of 2-categories $\bigoplus_{N, n \geq 0} 2\mathrm{Foam}_n(N) \to \mathfrak{F}$ which is the identity on objects, morphisms, and 2-morphisms.  
\end{proposition}

\begin{proof}
    Observe that the relations for $2\mathrm{Foam}_n(N)$ are just the Blanchet foam relations \cite{Blan} used in \cite{EST1}. Hence, the $2\mathrm{Foam}$ 2-categories are almost identical to the 2-category $\mathfrak{F}$, the only difference being that in taking the full sub 2-category in defining $\mathfrak{F}$ allows for the possibility of foams between upward directed webs that factor through non-upward directed webs.   Since we quotient by isotopy and every closed component can be evaluated, it is clear that every 2-morphism in $\mathfrak{F}$ is equal to one from $2\mathrm{Foam}$.
\end{proof}

\subsection{The categorified quantum group}

\begin{definition}[cf.\ Definition 2.1 and Section 2.3.3 of \cite{QR}] 
    The categorified quantum group $\cal{U}_Q(\mathfrak{gl}_n)$ is the $\mathbb{Z}$-graded $\mathbb{Z}$-linear 2-category defined as follows.
    \begin{itemize}
        \item The set of objects of $\cal{U}_Q(\mathfrak{gl}_n)$ is $\mathbb{Z}^n$; we refer to elements $\vec{k} \in \mathbb{Z}^n$ as $\mathfrak{gl}_n$-weights.
    
        \item 1-morphisms of $\cal{U}_Q(\mathfrak{gl}_n)$ are defined to be composable sequences of identity 1-morphisms $\mathbf{1}_{\vec{k}}$ (for objects $\vec{k} \in \mathbb{Z}^n$) and, for $1 \leq i \leq n-1$, basic 1-morphisms
        \[
            \mathbf{1}_{\vec{k}+\alpha_i}\cal{E}_i\mathbf{1}_{\vec{k}} \colon \vec{k} \to \vec{k} + \alpha_i, \qquad \mathbf{1}_{\vec{k}-\alpha_i}\cal{F}_i\mathbf{1}_{\vec{k}} \colon \vec{k} \to \vec{k} - \alpha_i
        \]
        where $\alpha_i \in \mathbb{Z}^n$ is $(0,\ldots,0,1,-1,0,\ldots,0)$ with $1$ in position $i$.
    
        \item 2-morphisms of $\cal{U}_Q(\mathfrak{gl}_n)$ are generated under vertical and horizontal composition by the basic 2-morphisms listed below:
        
    \begin{centering}
    \[
        \raisebox{-.5\height}{\includegraphics{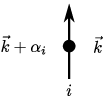}}\colon \mathbf{1}_{\vec{k}+\alpha_i}\cal{E}_i\mathbf{1}_{\vec{k}} \to \mathbf{1}_{\vec{k}+\alpha_i}\cal{E}_i\mathbf{1}_{\vec{k}}, \quad \raisebox{-.5\height}{\includegraphics{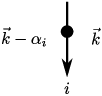}}\colon \mathbf{1}_{\vec{k}-\alpha_i}\cal{F}_i\mathbf{1}_{\vec{k}} \to \mathbf{1}_{\vec{k}-\alpha_i}\cal{F}_i\mathbf{1}_{\vec{k}},
    \]
    \[
        \raisebox{-.5\height}{\includegraphics{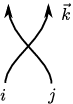}} \colon \cal{E}_i \cal{E}_j \mathbf{1}_{\vec{k}} \to \cal{E}_j \cal{E}_i \mathbf{1}_{\vec{k}}, \qquad \qquad \raisebox{-.5\height}{\includegraphics{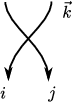}}\colon \cal{F}_i \cal{F}_j \mathbf{1}_{\vec{k}} \to \cal{F}_j \cal{F}_i \mathbf{1}_{\vec{k}},
    \]
    \[
        \raisebox{-.5\height}{\includegraphics{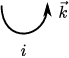}} \colon \mathbf{1}_{\vec{k}} \to \cal{F}_i \cal{E}_i \mathbf{1}_{\vec{k}}, \qquad \raisebox{-.5\height}{\includegraphics{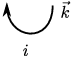}}\colon\mathbf{1}_{\vec{k}} \to \cal{E}_i \cal{F}_i \mathbf{1}_{\vec{k}},
    \]
    \[
        \raisebox{-.5\height}{\includegraphics{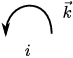}} \colon \cal{F}_i \cal{E}_i  \mathbf{1}_{\vec{k}} \to \mathbf{1}_{\vec{k}}, \qquad \raisebox{-.5\height}{\includegraphics{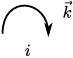}}\colon  \cal{E}_i \cal{F}_i \mathbf{1}_{\vec{k}} \to \mathbf{1}_{\vec{k}}.
    \]
    \end{centering}
    
    The dot maps have degree $2$; the crossing maps have degree $-a_{ij}$ where 
    \[
         a_{ij} =   \begin{cases}
                        2 & \text{if $i=j$}, \\
                        -1& \text{if $|i-j|=1$}, \\
                        0 & \text{otherwise}. 
                    \end{cases}
    \]
    The cup and cap maps on the left have degree $1 + k_i - k_{i+1}$, while the ones on the right have degree $1 - k_i + k_{i+1}$. We then impose the relations in items 1 through 6 of \cite[Definition 2.1]{QR}. In these relations, one should take $t_{i,i+1} = -1$, $t_{i+1,i} = 1$, and $t_{i,j} = 1$ for $|i-j| \geq 2$; note that with these choices of $t_{i,j}$ the definition in \cite{QR} works over $\mathbb{Z}$ as well as over a field.
    \end{itemize}
\end{definition} 

\begin{remark}
    As mentioned in \cite[Section 2.3.3]{QR}, for any $N$, one can identify the 2-category $\cal{U}_Q(\mathfrak{sl}_n)$ as defined in \cite[Definition 2.1]{QR} with the full sub-2-category of $\cal{U}_Q(\mathfrak{gl}_n)$ as defined above on objects $\vec{k}$ whose sum of entries is $N$. 
    The translation between $\mathfrak{sl}_n$ weights in $\mathbb{Z}^{n-1}$ and $\mathfrak{gl}_n$ weights in $\mathbb{Z}^n$, given $N$, associates to an $\mathfrak{sl}_n$ weight $\lambda = (\lambda_1,\ldots,\lambda_{n-1})$ the unique $\mathfrak{gl}_n$ weight $\vec{k}$ such that $\lambda_i = k_i - k_{i+1}$ for $1 \leq i \leq m-1$ and $\sum_{i=1}^n k_i = N$.
\end{remark}

\begin{remark}
    Following Remark~\ref{rem:SuppressQuantumGradings}, we will suppress mention of the grading on 2-morphisms in $\cal{U}_Q(\mathfrak{gl}_n)$ below.
\end{remark}

\begin{proposition}[cf.\ Lemma 3.7 and Theorem 3.9 of \cite{QR}]\label{prop:FoamationFunctor}
    There is a ``foamation'' 2-functor 
    \[
        \Phi_2 \colon \cal{U}_Q(\mathfrak{gl}_n) \to \bigoplus_N 2\mathrm{Foam}_n(N)
    \]
    defined as follows.
    \begin{itemize}
        \item For objects of $\cal{U}_Q(\mathfrak{gl}_n)$, viewed as $\mathfrak{gl}_n$ weights $\vec{k} \in \mathbb{Z}^n$, if all entries of $\vec{k}$ are in $\{0,1,2\}$ then $\Phi_2$ sends $\vec{k}$ to $\vec{k}$ viewed as an object of $2\mathrm{Foam}_n(N)$ where $N$ is the sum of entries of $\vec{k}$. If $\vec{k}$ has any entries outside $\{0,1,2\}$, then $\Phi_2$ sends $\vec{k}$ to the zero object of $2\mathrm{Foam}_n(N)$. %
        \item $\Phi_2$ sends basic 1-morphisms of the form $\cal{E}_i$ or $\cal{F}_i$ to the ladder webs specified in \cite[equation (2.9)]{QR}, where $a_i$ in \cite{QR} is our $k_i$. Our first coordinate is drawn right to left and our second coordinate is drawn top to bottom in \cite[equation (2.9)]{QR}. 
        \item $\Phi_2$ sends generating 2-morphisms in $\cal{U}_Q(\mathfrak{gl}_n)$ to the foams specified in \cite[Lemma 3.7 and Theorem 3.9]{QR}.
    \end{itemize}
\end{proposition}

\begin{proof}
    Theorem 3.9 in \cite{QR} as stated gives, for any fixed $N \in \mathbb{Z}$, a 2-functor from $\check{\cal{U}}_Q(\mathfrak{sl}_n)$ to $2\mathrm{Foam}_n(N)$, where $\check{\cal{U}}_Q(\mathfrak{sl}_n)$ contains $\cal{U}_Q(\mathfrak{sl}_n)$ as a full sub-2-category and also has 1-morphisms for divided powers of $\cal{E}_i$ and $\cal{F}_i$ (we will not need these here). Restricting from $\check{\cal{U}}_Q(\mathfrak{sl}_n)$ to $\cal{U}_Q(\mathfrak{sl}_n)$, summing over all $N$, and identifying the summed domain with $\cal{U}_Q(\mathfrak{gl}_n)$, we get the 2-functor in the statement of the proposition.
\end{proof}

Since $\bigoplus_N 2\mathrm{Foam}_n(N)$ is a sub-2-category of $\mathfrak{F}$, we can view the above proposition as giving a 2-functor from $\cal{U}_Q(\mathfrak{gl}_n)$ to $\mathfrak{F}$. 

\section{A signed Burnside lift of Blanchet--Khovanov algebras}\label{sec:SignedBurnsideLift}

\begin{definition}\label{def:ShapeMulticatForWk}
    Let $\Sc^0_{\vec{k}}$ be the shape multicategory of $W_{\vec{k}}$, and let $\Sc_{\vec{k}}$ be the canonical groupoid enrichment of $\Sc^0_{\vec{k}}$.
\end{definition}

Recall $W_{\vec k}$ from \Cref{def:rlFg}. Given $u,v \in W_{\vec{k}}$, let $v^*$ denote $v \subset [0,1] \times (0,1)$ reflected in the first coordinate and with orientations reversed. Then let $uv^*$ be the result of gluing $u \subset [0,1/2] \times (0,1)$ and $v^* \subset [1/2,1] \times (0,1)$; we can view $uv^*$ as a $\mathfrak{gl}_2$-web from $2\omega_{n,m}$ to itself, as on the left of \Cref{fig:ClosingWeb}.

For any $\mathfrak{gl}_2$-web $w$ from $2 \omega_{n,m}$ to itself, connect the 2-labeled edges of $w$ at $\{1\} \times (0,1)$ to the 2-labeled edges at $\{0\} \times (0,1)$ by circling around to the left as in Figure~\ref{fig:ClosingWeb}; the result is a closed web $\overline{w}$. 

\begin{figure}
    \centering
    \includegraphics[scale=0.7]{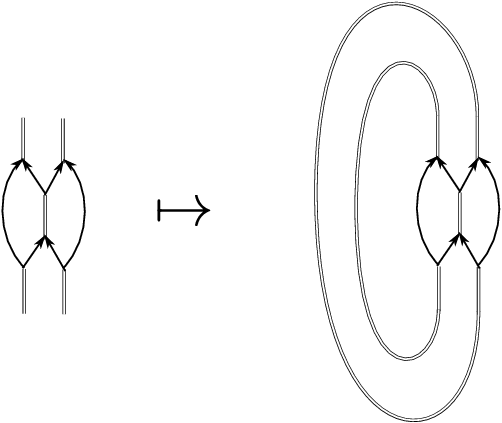}
    \caption{Closing a $\mathfrak{gl}_2$-web $w$ to the left; we are drawing the first coordinate bottom to top and the second coordinate left to right.}
    \label{fig:ClosingWeb}
\end{figure}

Now, given $u$ and $v$, choose any $\mathfrak{gl}_2$-foam $F$ from $\varnothing$ to $\overline{uv^*}$ whose underlying topological foam $\hat{F}$ consists of $k$ disks bounding the $k$ circles of the topological web underlying $\overline{uv^*}$. Each disk $D_i$ of $\hat{F}$ is divided into some number of facets by the 2-labeled facets of $F$; for each $i$, choose a preferred facet of $D_i$ in $F$. Then the $2^k$ ways of choosing dot or no-dot on the preferred facet of each $D_i$ give us a basis for the abelian group of $\mathfrak{gl}_2$-foams bounding $\overline{uv^*}$.

\begin{definition} 
    The set of $2^k$ possible labelings on the chosen foam $F$ is called $\mathrm{Foam}_{\mathrm{chosen}}(u,v)$.
\end{definition}

\begin{remark}\label{rem:ClosedVs2Closed}
    By applying homeomorphisms of the cube, we can view $\mathfrak{gl}_2$-foams from $u$ to $v$ as $\mathfrak{gl}_2$-foams bounding $\overline{uv^*}$ and vice-versa. See Figure~\ref{fig:ClosingFoam}, which shows how to go from a foam $u \to v$ to a foam $\varnothing \to \overline{uv^*}$ in two steps. The first step unfolds the left and top faces of the cube, and the second step unfolds the front and back faces in the resulting cube. Both steps are homeomorphisms from the cube to itself and are thus reversible. Note that the orientations on the back face of the cube in a foam from $u$ to $v$ get reversed when viewing the foam as going from $\varnothing$ to $uv^*$; the orientations on all other faces of the cube are preserved.
    
    In particular, we can view $\mathrm{Foam}_{\mathrm{chosen}}(u,v)$ as a set of foams from $u$ to $v$. Forgetting 2-labeled facets gives a bijection from $\mathrm{Foam}_{\mathrm{chosen}}(u,v)$ to $\mathrm{Disks}(\hat{u},\hat{v})$.
\end{remark}

\begin{figure}
    \centering
    \includegraphics[width=\textwidth]{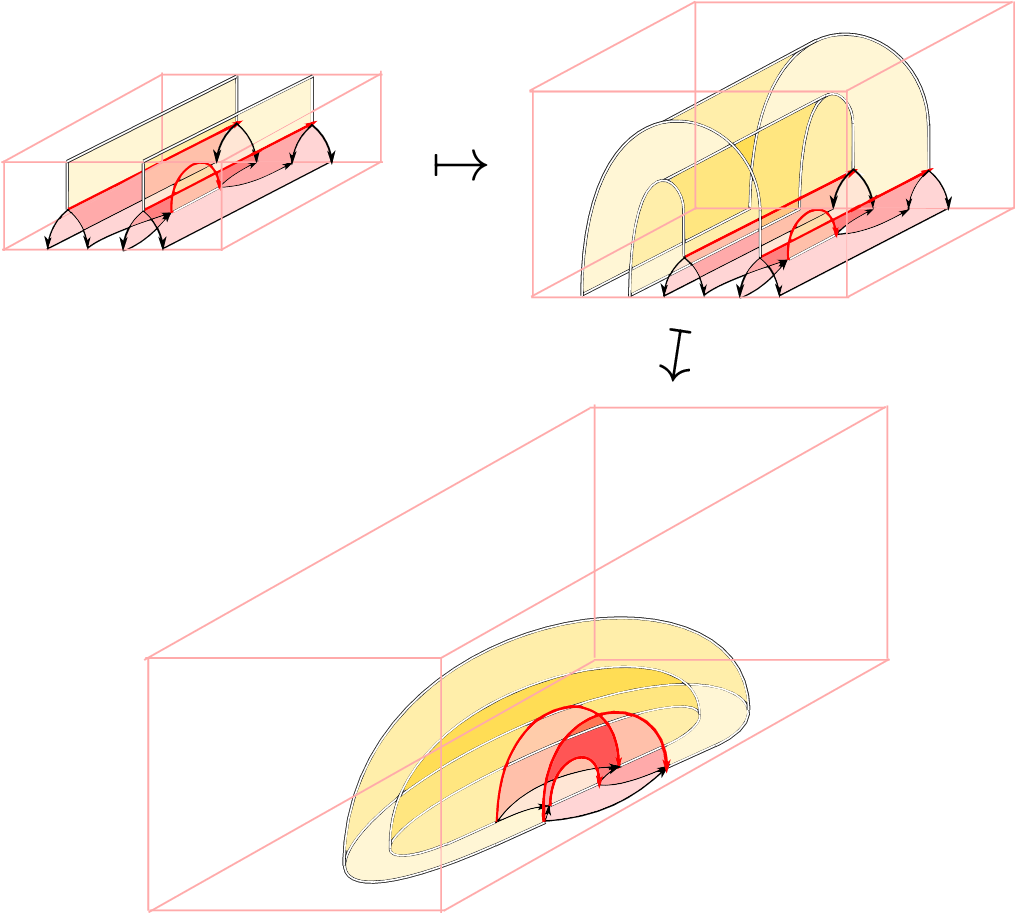}
    \caption{First step: unfolding the top and left faces. The top, left, and bottom faces of the first cube become the bottom face of the second cube. Second step: unfolding the front and back faces; the front, bottom, and back faces of the second cube become the bottom face of the third cube. Orientations get reversed on the back face in the second step.}
    \label{fig:ClosingFoam}
\end{figure}

\begin{definition}\label{def:MultifunctorGl2Case}
    Define a multifunctor 
    \[
        \Phi_{\vec{k}} \colon \Sc_{\vec{k}} \to \underline{\mathscr{B}}_{\sigma}
    \]
as follows.
\begin{itemize}
    \item An object $(u,v)$ of $\Sc_{\vec{k}}$ gets sent to  the finite set  $\mathrm{Foam}_{\mathrm{chosen}}(u,v)$.

    \item A 1-multimorphism of $\Sc_{\vec{k}}$ from $(u_0,u_1),\ldots,(u_{m-1},u_m)$ to $(u_0,u_m)$, is a tree $T$ as in \Cref{def:MultifunctorSl2Case}. We want to send $T$ to a signed correspondence; note that  \Cref{def:MultifunctorSl2Case} gives us an unsigned correspondence $\hat{A}$ from 
    \[
        \mathrm{Disks}(\hat{u}_0,\hat{u}_1) \times \cdots \times \mathrm{Disks}(\hat{u}_{m-1},\hat{u}_m)
    \]
    to $\mathrm{Disks}(\hat{u}_0, \hat{u}_m)$, which we can view as an unsigned correspondence from 
    \[
        \mathrm{Foam}_{\mathrm{chosen}}(u_0,u_1) \times \cdots \times \mathrm{Foam}_{\mathrm{chosen}}(u_{m-1},u_m)
    \]
    to $\mathrm{Foam}_{\mathrm{chosen}}(u_0,u_m)$. 
    
    We upgrade $\hat{A}$ to a signed correspondence $A$ by, for each nonempty matrix entry of $\hat{A}$ (say of size $2^k)$, looking at the corresponding entry of the algebraic matrix $M$ for the $m$-fold multiplication in $\mathfrak{W}^{\circ}_{\vec{k}}$ with respect to our chosen bases $\mathrm{Foam}_{\mathrm{chosen}}(u_{i-1},u_i)$. This multiplication concatenates the cubes the basis elements live in; we can alternatively view it in terms of a foamy variant of the frames of   \Cref{def:Frames}. An example $F_{u_0,\ldots,u_m}$ of such a frame is shown on the left of \Cref{fig:FoamyFrame}. The notation $F_{u_0,\ldots,u_m}$ is slightly inaccurate; the orientations on the singular seams of $F_{u_0,\ldots,u_m}$ also depend on the orientations of singular seams in the $m$ basis elements we are multiplying together.
    
    As opposed to the frames of \Cref{def:Frames}, the frame $F_{u_0,\ldots,u_m}$ is embedded in a ``sheltered'' brick castle where additional slabs have been added to the back side and the top of the castle (see the right of Figure~\ref{fig:FoamyFrame}). Multiplying basis elements by concatenation is equivalent to plugging them into the slots of the corresponding frame, like rods in a nuclear reactor.

    \begin{figure}
        \centering
        \includegraphics[width=\textwidth]{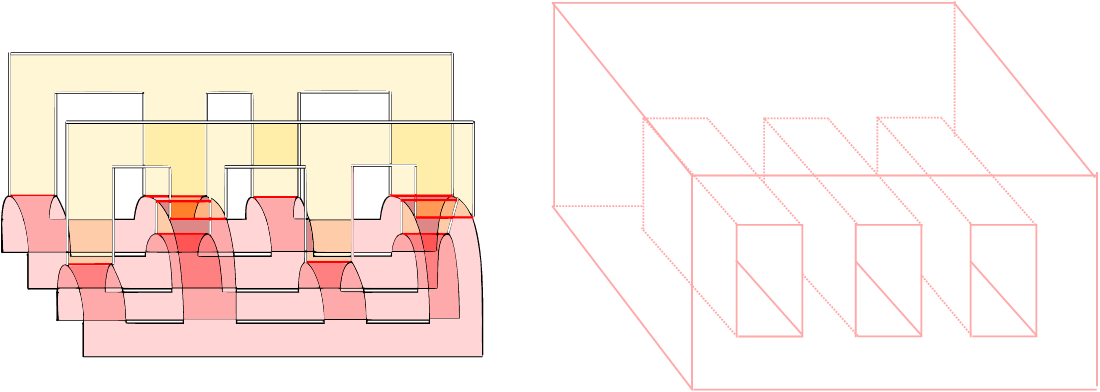}
        \caption{Left: a foamy frame; orientations on the singular seams are determined by the basis elements being plugged into the $m$ ``slots'' of the frame. Right: the sheltered brick castle in which the foamy frame is embedded.}
        \label{fig:FoamyFrame}
    \end{figure}
    
    Equivalently, we can perform this multiplication in the following series of steps:
    \begin{enumerate}
        \item Take the elements of $\mathrm{Foam}_{\mathrm{chosen}}(u_{i-1},u_i)$ and ``tilt them over to the left'' as in the first step of Figure~\ref{fig:ClosingFoam}, by unfolding the left side of the cube in which the foams are embedded. At this point we could perform the multiplication by gluing the foams into an unsheltered frame like the one on the left of Figure~\ref{fig:BasicFrame}; the frame will have some 2-labeled facets. Alternatively, we could continue and take the multimerge cobordism perspective.
        \item As in the second step of Figure~\ref{fig:ClosingFoam}, unfold the front and back of the cube to turn the elements of $\mathrm{Foam}_{\mathrm{chosen}}(u_{i-1},u_i)$ into $\mathfrak{gl}_2$-foams bounding closed webs. Take the disjoint union of the resulting foams by gluing their embedding cubes front-to-back in order.
        \item\label{it:ZipUnzipSaddle} View the foamy frame $F_{u_0,\ldots,u_m}$ as a multimerge foam as in Figure~\ref{fig:TransformingFoamyFrame}; as such, it can be written as a composition of zip, unzip, and saddle foams (e.g.\ the saddle in (21) in \cite{EST1} is the composition of two unzips and a zip). Apply the linear maps from these zip, unzip, and saddle foams, in order, to our disjoint-union foam. 
        \item The result is an element of the abelian group of $\mathfrak{gl}_2$-foams from the empty set to $u_0 u_m^*$, which we can express in our basis $\mathrm{Foam}_{\mathrm{chosen}}(u_0,u_m)$. 
    \end{enumerate}
    We can view the result of the first two steps as a set $\beta_{\mathrm{chosen}}$ of foams from the empty set to $\overline{u_0 u_1^*} \sqcup \cdots \sqcup \overline{u_{m-1}u_m^*}$ whose underlying topological foams are representatives for the usual $2^k$ isotopy classes of dotted disks bounding the underlying topological web of $\overline{u_0 u_1^*} \sqcup \cdots \sqcup \overline{u_{m-1}u_m^*}$. We can also choose an admissible flow (in the sense of \cite[Definition 3.4]{KW}) on $\overline{u_0 u_1^*} \sqcup \cdots \sqcup \overline{u_{m-1}u_m^*}$ and get another such set $\beta_{\mathrm{KW}}$ of $2^k$ foams, the (flow-dependent) Krushkal--Wedrich basis for the abelian group $\mathcal{F}$ of $\mathfrak{gl}_2$-foams bounding $\overline{u_0 u_1^*} \sqcup \cdots \sqcup \overline{u_{m-1}u_m^*}$ \cite[Definition~3.7]{KW}. 
    
    Both $\beta_{\mathrm{chosen}}$ and $\beta_{\mathrm{KW}}$ are such that forgetting 2-labeled facets produces the usual dotted-disks basis for the abelian group of dotted cobordisms bounding the underlying topological web of $\overline{u_0 u_1^*} \sqcup \cdots \sqcup \overline{u_{m-1}u_m^*}$. Thus, by \cite[Proposition 3.12]{KW}, there is a canonical bijection between $\beta_{\mathrm{chosen}}$ and $\beta_{\mathrm{KW}}$, and each element of $\beta_{\mathrm{chosen}}$ is plus or minus one times the corresponding element of $\beta_{\mathrm{KW}}$. It follows that the change-of-basis-of-$\mathcal{F}$ matrix $M_{\mathrm{CoB}}$ from $\beta_{\mathrm{chosen}}$ to $\beta_{\mathrm{KW}}$ has each column given by plus or minus a standard basis vector (zeroes in all entries but one, and one in the remaining entry). In particular, each column of $M_{\mathrm{CoB}}$ has its entries all nonnegative or all nonpositive.

    To get the full algebraic matrix $M$, one can take $M = M'' M' M_{\mathrm{CoB}}$ where $M'$ and $M''$ are defined as follows.
    Each zip, unzip, and saddle foam from item \eqref{it:ZipUnzipSaddle} above induces an admissible flow on its target, which has its own Krushkal--Wedrich basis. By \cite[Proposition 3.13]{KW}, the matrices for these zip, unzip, and saddle foams, in the Krushkal--Wedrich bases, have all entries of the entire matrix either all nonnegative or all nonpositive. We can take $M'$ to be the product of these matrices; the entries of $M'$ are all nonnegative or all nonpositive. Finally, we can take $M''$ to be the change-of-basis matrix from the Krushkal--Wedrich basis of the abelian group of $\mathfrak{gl}(2)$-foams bounding $\overline{u_0 u_m^*}$ to our chosen basis of this group of bounding foams; by \cite[Proposition 3.13]{KW}, each row of $M''$ has its entries either all nonnegative or all nonpositive.

    It follows that in producing $M$ by matrix multiplication as above, we never get any cancellation in the dot products used for matrix multiplication. Equivalently, if $M^+$ is the analogue of $M$ for the usual Khovanov algebras $H^p$ (where all entries of $M^+$ are nonnegative), then each entry of $M$ equals the corresponding entry of $M^+$ times some sign $\sigma \in \{\pm 1\}$. We define all of the elements of the entry of the signed correspondence $A$ associated to $T$ to have sign $\sigma$. 

    \item We send a change-of-tree 2-morphism $T \to T'$ in $\mathcal{S}_{\vec{k}}$ to the entrywise bijection of signed correspondences $A_T \to A_{T'}$ specified by the entrywise bijection $\phi_{T \to T'}$ of unsigned correspondences; we must check that $\phi_{T \to T'}$ sends positive entries of $A_T$ to positive entries of $A_{T'}$ and negative entries of $A_T$ to negative entries of $A_{T'}$. Indeed, all entries of $A_T$ have the same sign (the sign of all entries of $M_T$) and all entries of $A_{T'}$ have the same sign (the sign of all entries of $M_{T'}$). The two signs agree because $M_T$ and $M_{T'}$ are the same algebraic matrix for $m$-fold multiplication in $\mathfrak{W}^{\circ}_{\vec{k}}$ in our chosen bases, which is independent of the choice of tree because $\mathfrak{W}^{\circ}_{\vec{k}}$ is associative.

\end{itemize}
\end{definition}

\begin{figure}
    \centering
    \includegraphics[width=\textwidth]{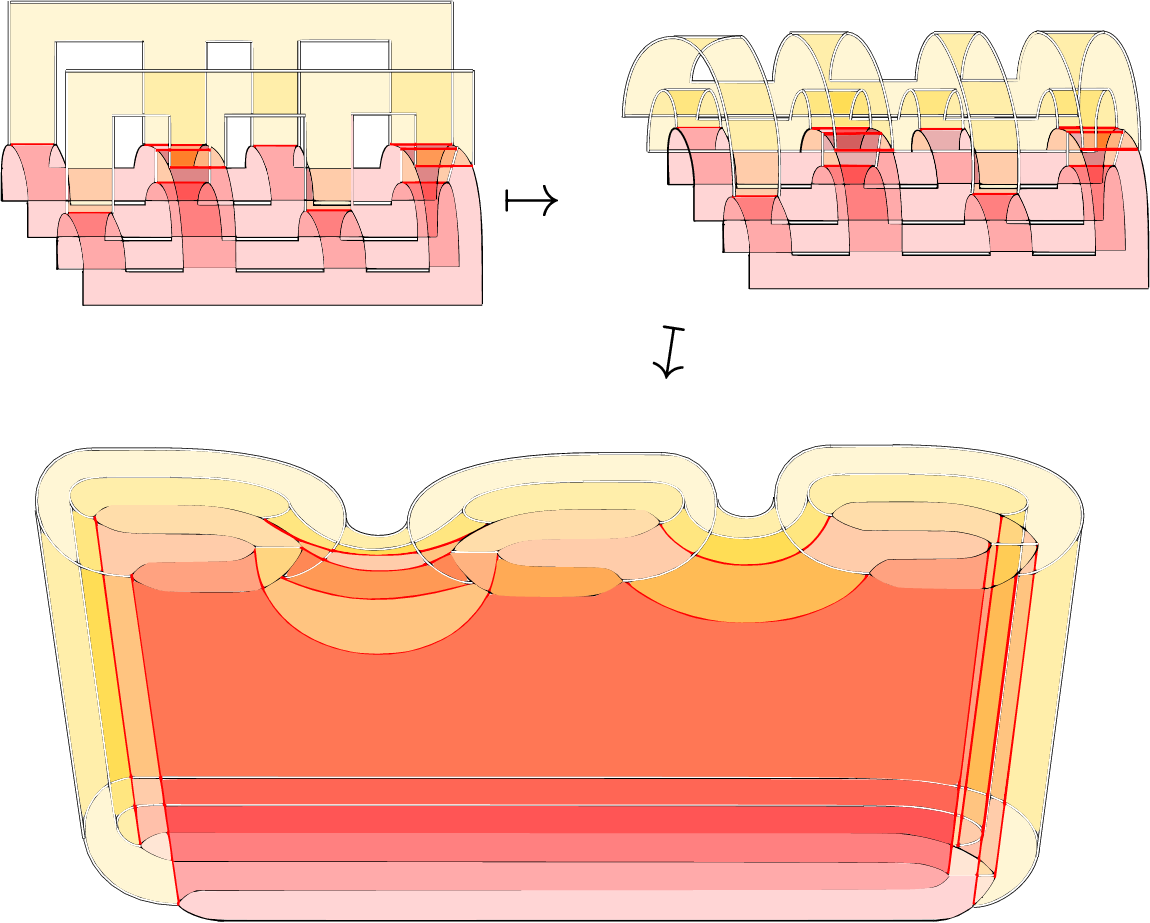}
    \caption{Transforming a foamy frame into a multimerge foam.}
    \label{fig:TransformingFoamyFrame}
\end{figure}

The fact that $\Phi_{\vec{k}}$ respects vertical composition and horizontal multicomposition of 2-morphisms follows from Theorem~\ref{thm:Sl2MultifunctorValid}, since Definition~\ref{def:MultifunctorGl2Case} agrees with Definition~\ref{def:MultifunctorSl2Case} on 2-morphisms. The fact that $\Phi_{\vec{k}}$ respects multicomposition of 1-multimorphisms follows from Theorem~\ref{thm:Sl2MultifunctorValid} and from how we upgrade unsigned correspondences in Definition~\ref{def:MultifunctorSl2Case} to signed correspondences in Definition~\ref{def:BKAlg} by getting the signs from algebraic matrices.

By construction, when we pass from $\underline{\mathscr{B}}_{\sigma}$ to $\underline{\mathsf{Ab}}$ by linearizing, the correspondences $A$ associated to $1$-multimorphisms in $\mathcal{S}_{\vec{k}}$ become the matrices $M$ for repeated multiplication in $\mathfrak{W}^{\circ}_{\vec{k}}$. Thus, when we linearize $\Phi_{\vec{k}} \colon \mathcal{S}_{\vec{k}} \to \underline{\mathscr{B}}_{\sigma}$, we obtain $\mathfrak{W}^{\circ}_{\vec{k}}$ viewed as a multifunctor from $\mathcal{S}_{\vec{k}}$ to $\underline{\mathsf{Ab}}$.

\begin{remark}\label{rem:WebBimodules}
    We could adapt Definition~\ref{def:MultifunctorGl2Case} to the bimodules associated to $\mathfrak{gl}_2$-webs, using foamy ``web frames'' analogous to the tangle frame shown in Figure~\ref{fig:GeneralFrame} (except that for $\mathfrak{gl}_2$-webs, which are flat, the frames will be more like the flat tangle frames of Definition~\ref{def:FlatTangleFrame} since they will not have saddles in their middle section). For brevity, we will omit the details.
\end{remark}

\begin{remark}\label{rem:GeneralMorSpaces}
    We have written everything above for webs from $2\omega_{n,m}$ to some $\vec{k} \in \bbl$. However, nothing would change if we replace $(2 \omega_{n,m}, \vec{k})$ with $(\vec{k}_1, \vec{k}_2)$ for two elements $\vec{k}_1, \vec{k}_2 \in \bbl$ with the same length $n$ and sum $N$ (now $N$ is not necessarily even). 
    
    In particular, given two objects $\vec{k}_1, \vec{k}_2$ of $\mathfrak{F}$, we get a signed Burnside version of the morphism category from $\vec{k}_1$ to $\vec{k}_2$ in $\mathfrak{F}$.

    More generally, we could allow finite sequences $\vec{k}_i$ of elements of $\{-2,-1,0,1,2\}$, possibly with different lengths and sums, and do everything for (not necessarily upward-directed) $\mathfrak{gl}_2$-webs from $\vec{k}_1$ to $\vec{k}_2$. We thereby get signed Burnside versions of the morphism categories in the 2-category $\overline{\mathfrak{F}}$ from \cite[Definition 2.17]{EST1}.
\end{remark}

\section{Spectral 2-representations}\label{sec:UpToSign2Action}

In this section, we show that Lawson--Lipshitz--Sarkar's spectrificiation of Khovanov's arc algebra can be interpreted as spectrifying a 2-representation of the categorified quantum group $\cal{U}_Q(\mathfrak{gl}_n)$.    Recall that  a \new{2-representation} of $\cal{U}_Q(\mathfrak{gl}_n)$ is a graded $k$-linear 2-functor $\cal{U}_Q(\mathfrak{gl}_n)\to \cal{K}$ for some graded additive 2-category $\cal{K}$.  It is most common to take $\Bbbk=\Z$ or a field; here we will take $k = \mathbb{F}_2$.

As explained in the introduction, Khovanov's arc algebras $H^m$ naturally give rise to a categorification of the invariant space ${\rm Inv}(V^{\otimes 2m})$ of tensor powers of the defining representation of $U_q(\mf{sl}_2)$, and these invariant spaces naturally organize into a representation of $U_q(\mathfrak{gl}_n)$. Correspondingly, the arc algebras $H^m$ organize into a 2-representation of $\cal{U}_Q(\mathfrak{gl}_n)$. Here we leverage Lawson--Lipshitz--Sarkar's study of the homotopy functoriality of spectral arc algebras to define what we call a spectral 2-representation lifting this algebraic 2-representation.

The target of our spectral 2-representations will be a bicategory of spectral categories, spectral bimodules, and spectral bimodule homomorphisms. To match with the constructions in \cite{LLS-homotopy}, we will work with a strictified version of this bicategory in which composition of 1-morphisms is strictly associative.\footnote{This strictification allows Lawson--Lipshitz--Sarkar to use the theory of enriched multicategories, with strictly associative multicomposition, when defining tangle multimodules over spectral arc algebras. For our purposes, consideration of bimodules, rather than multimodules, suffices.}

\subsection{A 2-category of spectral bimodules}

Recall that Lawson--Lipshitz--Sarkar \cite[Definition 4.20]{LLS-func} denote by $\mathrm{SBim}$ the multicategory enriched in spectral categories whose objects are spectral categories and whose multimorphism categories are (spectral) derived categories of spectral multimodules. We will only look at the 1-multimorphism categories, where objects are spectral bimodules over spectral categories $(\mathscr{A},\mathscr{B})$. Furthermore, we will work with an additive version of this 1-multimorphism category, namely the homotopy category of the model category of $(\mathscr{A},\mathscr{B})$-bimodules (see \cite[Proposition 2.4]{BlumbergMandell}).

\begin{definition}
    Let $\mathrm{SpecBim}$ denote the $\mathbb{Z}$-linear 2-category obtained by modifying Lawson--Lipshitz--Sarkar's $\mathrm{SBim}$ as above.
\end{definition}

Concretely, objects of $\mathrm{SpecBim}$ are finite graded spectral categories as in \cite[Definition 4.20]{LLS-func} and 1-morphisms from an object $\mathscr{A}$ to an object $\mathscr{B}$ are sequences of bimodules 
\[
    ({_{\mathscr{A}_1}}(\mathscr{M}_1)_{\mathscr{A}_2}, \ldots, {_{\mathscr{A}_{k-1}}}(\mathscr{M}_{k-1})_{\mathscr{A}_k})
\]
with $\mathscr{A}_1 = \mathscr{A}$ and $\mathscr{A}_k = \mathscr{B}$. 

For 2-morphisms from a sequence $(\mathscr{M}_1,\ldots,\mathscr{M}_{k-1})$ to a sequence $(\mathscr{M}'_1,\ldots,\mathscr{M}'_{k-1})$, let $\mathscr{M}$ be the tensor product of cofibrant replacements of the $\mathscr{M}_i$, and similarly for $\mathscr{M}'$ (using a functorial choice of cofibrant replacements). A 2-morphism from the first sequence to the second is a morphism from $\mathscr{M}$ to $\mathscr{M'}$, up to sign, in the homotopy category of the model category of spectral bimodules.

\subsection{A spectral 2-representation}

Let $\mathrm{SpecBim}^{\mathbb{F}_2}$ denote the $\mathbb{F}_2$-linear 2-category obtained from $\mathrm{SpecBim}$ by tensoring each abelian group of 2-morphisms with $\mathbb{F}_2$.  A \new{spectral 2-representation} is a 2-representation with target $\cal{K}=\mathrm{SpecBim}$ or $\cal{K}=\mathrm{SpecBim}^{\mathbb{F}_2}$. In particular, a spectral 2-representation of $\cal{U}_Q(\mathfrak{gl}_n)$ maps weights $\vec{k}$ of $\mf{gl}_n$ to spectral categories, 1-morphisms $\cal{E}_i\mathrm{1}_{\vec{k}}$ and $\cal{F}_j\mathrm{1}_{\vec{k}}$ to spectral bimodules, and generating 2-morphisms (dots, crossings, caps, and cups) to morphisms in the homotopy category of the model category of spectral bimodules.  

\begin{theorem} \label{thm:2rep}
Let $\Bbbk=\mathbb{F}_2$.  Spectral arc algebras give rise to a spectral 2-representation  \[
\cal{U}_Q(\mathfrak{gl}_n)^{\mathbb{F}_2} \to \mathrm{SpecBim}^{\mathbb{F}_2}
\]
as follows.
\begin{itemize}
        \item An object $\vec{k}$ of $\cal{U}_Q(\mathfrak{gl}_n)^{\mathbb{F}_2}$ with all entries in $\{0,1,2\}$ and with an even number of ones (say $2m$) gets sent to Lawson--Lipshitz--Sarkar's spectral category $\mathscr{H}^m$. If $\vec{k}$ does not have an even number of ones, or if it has any entries outside $\{0,1,2\}$, it gets sent to the zero spectral category.

        \item A 1-morphism $\mathcal{E}_i \mathbf{1}_{\vec{k}}$ or $\mathcal{F}_i \mathbf{1}_{\vec{k}}$ of $\cal{U}_Q(\mathfrak{gl}_n)^{\mathbb{F}_2}$, whose source and target objects have entries only in $\{0,1,2\}$, gets sent to Lawson--Lipshitz--Sarkar's spectral bimodule $\mathscr{X}(T)$ for the underlying topological web of the ladder web associated to the 1-morphism, viewed as a flat tangle $T$. If the source or target object of the 1-morphism has any entries outside $\{0,1,2\}$, then the 1-morphism gets sent to the corresponding zero bimodule. General 1-morphisms of $\cal{U}_Q(\mathfrak{gl}_n)^{\mathbb{F}_2}$ get sent to sequences of these bimodules.

        \item A generating 2-morphism $\alpha$ (locally a dot, crossing, cup, or cap) of $\cal{U}_Q(\mathfrak{gl}_n)^{\mathbb{F}_2}$ gets sent to a morphism between sequences of spectral bimodules. If any of the bimodules in either sequence is the zero bimodule, then $\alpha$ gets sent to the zero morphism in $\mathrm{SpecBim}^{\mathbb{F}_2}$. Otherwise, $\alpha$ is sent to a composition of 2-morphisms in $\mathrm{SpecBim}^{\mathbb{F}_2}$ that we now describe:

        \begin{itemize}
            \item From the tensor product of cofibrant replacements for the first sequence (each coming from some basic ladder web), map via the (iterated) gluing map from \cite[equation (4.6)]{LLS-func} to the the cofibrant replacement of the spectral bimodule associated to the composite of the ladder webs. More precisely, one should use the derived version of the gluing map, shown to be an equivalence of spectral bimodules in \cite[Theorem 5]{LLS-KST}.
            
        \item Apply the cofibrant replacement of the up-to-sign 2-morphism $\mathscr{X}(\Sigma)$ in $\mathrm{SpecBim}$ associated by Lawson--Lipshitz--Sarkar \cite[item (4) of Definition 4.24 and paragraph above Theorem 6]{LLS-func}) to the underlying topological foam of the image of $\alpha$ under the foamation 2-functor $\Phi_2$ of Proposition~\ref{prop:FoamationFunctor}, viewed as a cobordism $\Sigma$ between (flat) tangles. 

        \item Map from the cofibrant replacement of the bimodule for the output-side glued ladder web to the tensor product of cofibrant replacements for the second sequence via the inverse of the (iterated) gluing map from \cite[equation (4.6)]{LLS-func}, interpreted in the derived sense as above. Note that this inverse exists as a 2-morphism in $\mathrm{SpecBim}^{\mathbb{F}_2}$ because, by \cite[Theorem 5]{LLS-KST} (see also \cite[Lemma 4.23]{LLS-func}), the gluing map is an equivalence of spectral bimodules and thus gets inverted in the homotopy category of the model category of spectral bimodules.
        \end{itemize}

    \noindent Correspondingly, a general 2-morphism of $\cal{U}_Q(\mathfrak{gl}_n)^{\mathbb{F}_2}$ gets sent to a sum of compositions of 2-morphisms in $\mathrm{SpecBim}^{\mathbb{F}_2}$ associated to its basic pieces. 
    \end{itemize}
\end{theorem}

\begin{proof}
    We first claim that it suffices to show that for each defining relation in $\cal{U}_Q(\mathfrak{gl}_n)$, say an equality of 2-morphisms from some 1-morphism $X$ to another 1-morphism $Y$ (each of $\{X,Y\}$ being a composite of some $\cal{E}_i$ and $\cal{F}_j$), the relation holds between 2-morphisms in $\mathrm{SpecBim}^{\mathbb{F}_2}$ from the (single-term sequence) spectral bimodule associated to the composite ladder web for $X$ to the (single-term sequence) spectral bimodule associated to the composite ladder web for $Y$.

    Indeed, what we really want to show is that the relation holds between 2-morphisms in $\mathrm{SpecBim}^{\mathbb{F}_2}$ from the (multi-term) sequence of spectral bimodules associated to the factors of $X$ to the (multi-term) sequence of spectral bimodules associated to the factors of $Y$. To prove the claim, note that the equality we are aiming for is in general a sum of compositions of generating 2-morphisms (locally dots, crossings, cups, and caps, extended horizontally by the identity). Consider one term in this sum, and suppose it's a composition $\alpha_k \circ \cdots \circ \alpha_1$ of generating 2-morphisms.

    By definition, $\alpha_k \circ \cdots \circ \alpha_1$ gets sent to a 2-morphism that schematically has the form
    \[
    (\textrm{unglue}  \circ (\textrm{map from } \alpha_k) \circ \textrm{glue} ) \circ \cdots \circ (\textrm{unglue}  \circ (\textrm{map from } \alpha_1) \circ \textrm{glue} ).
    \]
    The intermediate compositions ``glue $\circ$ unglue'' are identity 2-morphisms in $\mathrm{SpecBim}^{\mathbb{F}_2}$, so we can rewrite the above composition as
    \[
    \textrm{unglue}  \circ (\textrm{map from } \alpha_k) \circ \cdots \circ (\textrm{map from } \alpha_1) \circ \textrm{glue}.
    \]
    We will show below that if we sum the portion ``$(\textrm{map from } \alpha_k) \circ \cdots \circ (\textrm{map from } \alpha_1)$'' of the above expression over all terms in the sum defining the $\cal{U}_Q(\mathfrak{gl}_n)^{\mathbb{F}_2}$ relation in question, we get zero (mod 2). Given this, the same is true when we precompose all terms with ``glue'' and postcompose all terms with ``unglue,'' proving the claim.

    To complete the proof, recall that ``map from $\alpha_i$'' in the above schema can be described by applying the foamation functor of Proposition~\ref{prop:FoamationFunctor} to $\alpha_i$, yielding a $\mathfrak{gl}_2$-foam between $\mathfrak{gl}_2$-webs, and then forgetting 2-labeled facets and edges to get a cobordism $\Sigma_i$ from a flat tangle $T_i$ to a flat tangle $T'_i$. The map from $\alpha_i$ is then the up-to-sign map between spectral bimodules $\mathscr{X}(T_i) \to \mathscr{X}(T'_i)$ associated to $\Sigma$ in \cite[item (4) of Definition 4.24 and paragraph above Theorem 6]{LLS-func}. Item (PMF-6) of \cite[Proposition 4.25]{LLS-func} implies that composing the maps from all $\alpha_i$ gives the same map as composing the cobordisms $\Sigma_i$ to obtain a cobordism $\Sigma$ from a flat tangle $T$ to a flat tangle $T'$, then taking the map $\mathscr{X}(\Sigma) \colon \mathscr{X}(T) \to \mathscr{X}(T')$ associated to $\Sigma$.

    Now, for any defining relation of $\mathcal{U}_Q(\mathfrak{gl}_n)^{\mathbb{F}_2}$, one can check (although there are several cases to consider) that the cobordisms $\Sigma$ from the terms of the relation, formally summed, form an instance of the $\mathbb{F}_2$-coefficient Bar-Natan local relations for dotted cobordisms (these are: neck cutting, sphere with $k$ dots is zero if $k \neq 1$ and can be removed otherwise, two dots on the same connected component give zero). It thus suffices to show that when such a formal $\mathbb{F}_2$-linear combination of cobordisms $\Sigma$ forms an instance of the Bar-Natan relations, the corresponding linear combination of 2-morphisms $\mathscr{X}(\Sigma)$ in $\mathrm{SpecBim}^{\mathbb{F}_2}$ is zero.

    First note that by \cite[Theorem 6]{LLS-func}, since we only care about the 2-morphisms $\mathscr{X}(\Sigma)$ up to sign, we only need to care about the cobordisms $\Sigma$ up to isotopy. By \cite[Proposition 7.2]{LLS-func}, the maps $\mathscr{X}(\Sigma)$ also satisfy the Bar-Natan neck cutting relation over $\mathbb{F}_2$.
    
    Suppose that some $\Sigma$ has a connected component that is a closed sphere with $k$ dots. After an isotopy, we can apply item (PMF-4) of \cite[Proposition 4.25]{LLS-func} (specifically, commutativity of the diagram of \cite[(4.3)]{LLS-func}) to localize to a cobordism from the empty flat tangle to itself consisting of a single closed sphere with $k$ dots. By \cite[Lemma 7.3]{LLS-func}, the $\mathbb{F}_2$ Bar-Natan relations for closed spheres hold for the map associated to this local cobordism. Thus, they hold for $\Sigma$.
    
    Finally, if some $\Sigma$ has two dots on the same connected component, after an isotopy we can assume the dots are both on a sphere that is joined to the rest of $\Sigma$ by a neck. The neck-cutting relation implies that, over $\mathbb{F}_2$, $\mathscr{X}(\Sigma)$ is a sum of a map associated to a cobordism with a two-dot closed sphere and a map associated to a cobordism with a three-dot closed sphere. Both of these maps are zero by the previous case, proving the theorem.
\end{proof}

\begin{remark}
We could have stated a slightly stronger result than Theorem~\ref{thm:2rep}.  The most accurate statement is that some version of relations in the categorified quantum group $\cal{U}_Q(\mathfrak{gl}_n)$ hold in $\mathrm{SpecBim}$ where each term in a linear combination of 2-morphisms has been multiplied individually by $\pm 1$.
\end{remark}

\bibliographystyle{alpha}
\bibliography{bib_clean}

\end{document}